\documentclass{amsart}
\usepackage{amssymb}
\usepackage{graphicx}
\usepackage[colorlinks,bookmarks]{hyperref}
\usepackage{tikz}
\usetikzlibrary{arrows,positioning,shapes}
\usepackage[section]{placeins}
\usepackage[mathscr]{euscript}
\usepackage{enumitem}
\usepackage[all,cmtip]{xy}

\newtheorem{lemma}{Lemma}[subsection]
\newtheorem{theorem}[lemma]{Theorem}
\newtheorem{proposition}[lemma]{Proposition}
\newtheorem{corollary}[lemma]{Corollary}

\newtheorem{conjecture}[lemma]{Conjecture}
\newtheorem{introthm}{Theorem}

\newtheorem{introcor}[introthm]{Corollary}
\theoremstyle{definition}
\newtheorem{example}[lemma]{Example}
\newtheorem{remark}[lemma]{Remark}
\newtheorem*{notation}{Notation}
\newtheorem{definition}[lemma]{Definition}
\newtheorem{question}[lemma]{Question}

\newcommand{\G}[1][n]{\sym[#1][G]}		% wreath product
\newcommand{\A}[1][n]{\mathscr{A}_{#1}}		% arrangement
\newcommand{\M}[1][n]{\mathscr{M}_{#1}}		% complement
\renewcommand{\P}[1][n]{\mathscr{P}_{#1}}	% poset of layers
\newcommand{\D}[1][n]{\mathscr{D}_{#1}}		% Dowling poset
\newcommand{\DD}[1][n]{\mathscr{D}_{#1}^\vee}
\newcommand{\Q}{\mathscr{Q}}			% partition lattices
\renewcommand{\O}{\mathscr{O}}			% orbit set
\newcommand{\Z}{\mathbb{Z}}			% integers option 2
\newcommand{\QQ}{\mathbb{Q}}			% rationals
\newcommand{\C}{\mathbb{C}}			% complex numbers option 2
\renewcommand{\AA}{\mathbb{A}}			% affine space
\newcommand{\R}{\mathbb{R}}			% real numbers option 2

\newcommand{\sym}[1][n]{\mathfrak{S}_{#1}}	% symmetric group option 2
\newcommand{\n}[1][n]{\mathbf{#1}}		% set {1,2,...,n}
\newcommand{\oo}{\mathfrak{o}}			% an orbit

\newcommand{\zero}{\hat{0}}			% minimum of poset
\newcommand{\onehat}{\hat{1}}
\renewcommand{\H}{\widetilde{\Ho}}			% Whitney homology
\DeclareMathOperator{\WH}{WH}			% Whitney homology
\DeclareMathOperator{\Ho}{H}			% homology/cohomology

\newcommand{\wt}[1]{\widetilde{#1}}		% enriched partition
\newcommand{\parts}[1][\beta]{(\wt{#1},z)}%	more tuples

\newcommand{\into}{\hookrightarrow}		% injection

\DeclareMathOperator{\Conf}{Conf}		% config space
\DeclareMathOperator{\Aut}{Aut}			% automorphisms
\DeclareMathOperator{\rk}{rk}			% rank
\DeclareMathOperator{\Ind}{Ind}			% Induced representations
\DeclareMathOperator{\im}{im}			% image of a map

\title{Combinatorics of orbit configuration spaces}
\author{Christin Bibby}
\address{Department of Mathematics, Louisiana State University, Baton Rouge, LA, USA}
\email{\href{mailto:bibby@lsu.edu}{bibby@lsu.edu}}
\author{Nir Gadish}
\address{Department of Mathematics, Massachusetts Institute of Technology, Cambridge, MA, USA} 
\email{\href{mailto:ngadish@mit.edu}{ngadish@mit.edu}}

 %Arrangements of points, flats, hyperplanes

\begin{document}

\begin{abstract}
	From a group action on a space, define a variant of the configuration
space by insisting that no two points inhabit the same orbit. When the action is
almost free, this ``orbit configuration space'' is the complement of an
arrangement of subvarieties inside the cartesian product, and we use this
structure to study its topology.

	We give an abstract combinatorial description of its poset of layers
(connected components of intersections from the arrangement) which turns out to
be of much independent interest as a generalization of partition and Dowling
lattices. The close relationship to these classical posets is then exploited to
give explicit cohomological calculations. 
	
\end{abstract}

\maketitle

\section{Introduction}\label{sec:intro}

\subsection{Orbit configuration spaces}
A fundamental topological object attached to a topological space $X$ is its
ordered configuration space $\Conf_n(X)$ of $n$ distinct points in $X$. 
Analogously, given a group $G$ acting freely on $X$ one defines the 
\emph{orbit configuration space} by
\[\Conf_n^G(X) = \{(x_1,\dots,x_n)\in X^n\ |\ Gx_i\cap
Gx_j=\emptyset \text{ for } i\neq j\}.\]
These spaces were first defined in \cite{Xthesis} and
come up in many natural topological contexts, including:
\begin{itemize}
	\item Universal covers of $\Conf_n(X)$ when $X$ is a manifold with 
$\dim(X)>2$ \cite{Xthesis}.
	\item Classifying spaces of well studied groups, such as normal subgroups of surface braid groups with quotient $G^n$ \cite{Xthesis}.
	\item Arrangements associated with root systems
\cite{bibby2,looijenga,Moci2008}.
	\item Equivariant loop spaces of $X$ and $\Conf_n(X)$ \cite{X}.
\end{itemize}
A fundamental problem is thus to compute the cohomology
of $\Conf_n^G(X)$.
This has been previously studied e.g. by \cite{Casto2016,DS, FZ}.

The current literature typically requires the action to be free, with main
results relying on this assumption. 
For an action that is not free, one could simply throw out the set of singular points for the action
and consider $\Conf_n^G(X\setminus S)$, where 
\[ S := \operatorname{Sing}_G(X) = \bigcup_{g\in G\setminus\{e\}} X^g,\]
the set of points fixed by a nontrivial group element. However, the excision can
create more harm than good: e.g. when $X$ is a smooth projective variety,
removing $S$ destroys the projective structure and causes mixing of Hodge
weights in cohomology. In particular, having a projective structure makes a
spectral sequence calculation more manageable (see Theorem \ref{thm:spectral_sequence} and \S\ref{sec:specseq}). 
Furthermore, one is often interested in allowing 
orbit configurations to inhabit $S$, e.g. in arrangements arising from type B and D root systems (see \S\ref{sec:subarrangements}).

We propose an alternative approach: observe that inside $X^n$, the orbit
configuration space $\Conf_n^G(X\setminus S)$ is the complement of an 
arrangement $\A(G,X)$ of subspaces. 
In the case that $G$ is the trivial group, this arrangement is the set of diagonals $x_i=x_j$ in $X^n$.
The cohomology $\Ho^*(\Conf^G_n(X\setminus S))$ can then be computed from the combinatorics
of this arrangement and from $\Ho^*(X)$. Furthermore, the wreath product of $G$
with the symmetric group $\sym$, which we denote by $\G$, 
acts naturally on the space $X^n$, and this induces an $\G$--action on both
$\A(G,X)$ and its complement $\Conf_n^G(X\setminus S)$. The induced action on 
$\Ho^*(\Conf^G_n(X\setminus S))$ can also be traced
through the combinatorial computation. 

In the sequel \cite{BG2}, the authors study this action on cohomology through the lens of representation stability. The present paper focuses on the combinatorial foundations of orbit configuration spaces and machinery for computing their cohomology and other topological invariants.

\subsection{Running assumptions and notation}
For our study, we from hereon assume that $G$ and $S$ are finite sets, so that
the arrangement $\A(G,X)$ is finite.
Moreover, by a ``space" $X$ we mean either a CW complex or an algebraic variety over an algebraically 
closed field.

When discussing cohomology below, we will always suppress the coefficients and use the following convention. For $X$ a CW complex, let the cohomology $\Ho^*(X)$ denote singular cohomology with coefficients in any ring $R$; and for $X$ an algebraic variety, $\Ho^*(X)$ denotes $\ell$-adic cohomology with coefficients in either $\mathbb{Z}_\ell$ or $\QQ_\ell$.

\subsection{Combinatorics}

The combinatorics at play is the \emph{poset of layers}: connected components of
intersections from $\A(G,X)$, ordered by reverse inclusion. This poset admits an
abstract combinatorial description, that does not in fact depend on $X$ (only
depending on the $G$--set $S$) and it is of much independent interest. For example,
\begin{itemize}
	\item In the case of classical configuration spaces ($G$ trivial), the
poset is the lattice 
$\Q_{\n}$
of set partitions of $\n = \{1,2,\dots,n\}$.
	\item In the case that $G$ is a cyclic group acting on $X=\C$ via
multiplication by roots of unity, the poset is an instance of the Dowling lattice 
$\D(G)$, 
described in \cite{Dowling1973} as an analogue of the partition lattice which
consists of partial $G$--partitions of $\n$.
\end{itemize}
In \S\ref{sec:posetdefs}, we define the poset $\D(G,S)$ which specializes
to these classical examples and discuss the natural action of the wreath product
group $\G$.

Even though $\D(G,S)$ is not in general a lattice, it supports a myriad of
properties that have been fundamental in the modern study of posets, since it is
essentially built out of partition and Dowling lattices as indicated in the
following theorem (Theorem \ref{thm:intervals}):
\begin{introthm}[\textbf{Local structure of $\D(G,S)$}]\label{thm:intervals1}
	For any $\alpha,\beta\in \D(G,S)$ with $\alpha<\beta$, 
the interval $[\alpha,\beta]$ is isomorphic to a product
	\[ \Q_{\n[n_1]}\times \ldots \times \Q_{\n[n_d]} \times
\D[m_1](G_1)\times \ldots \times \D[m_k](G_k)  \]
where $\Q_{\n[n_i]}$ denotes a partition lattice and $\D[m_j](G_j)$ denotes a 
Dowling lattice for some subgroup $G_j\leq G$.
	In particular, every interval is a geometric lattice and has the
homology of a wedge of spheres.
\end{introthm}    

In the remainder of Section \S\ref{sec:combinatorics} we study the structure of these posets: In \S\ref{sec:functoriality} we discuss their functoriality in the various inputs; in \S\ref{sec:posetprops} we discuss local structure and prove Theorem \ref{thm:intervals1}. Of particular interest is \S\ref{sec:charpoly}, where we discuss the characteristic polynomial: a fundamental invariant of a ranked poset, which is a common generalization of the chromatic polynomial of a graph, and the Poincar\'{e} polynomial of the complement of a hyperplane arrangement. 
We give  
the following
factorization of the characteristic polynomial into linear factors, generalizing a long list of special cases stretching back to Arnol'd and Stanley's work on the pure braid group and the partition lattice.
\begin{introthm}[\textbf{Characteristic polynomial}]
Let $S\neq \emptyset$ be a $G$-set. Then
\[
\chi(\D(G,S);t) := 
\sum_{\beta\in\D(G,S)}\mu(\zero,\beta)t^{n-\rk(\beta)} 
= \prod_{i=0}^{n-1} (t-|S|-|G|i),
\]
where $\mu$ is the M\"obius function of the poset and $\zero$ is the minimum element.
An analogous factorization for the case $S=\emptyset$ appears in Theorem \ref{thm:charpoly} below.
\end{introthm}

Lastly, in \S\ref{sec:orbits}, we consider the action of $\G$ on the poset $\D(G,S)$, and describe its orbits. In \S\ref{sec:whitney}, we study their Whitney homology as a representation of $\G$: this invariant has proved important both for topology and for the abstract theory of posets, and will be later used when discussing orbit configuration spaces in \S\ref{sec:arrangements}.
In \cite{paolini}, Paolini further studies the homology of $\D(G,S)$, finding more properties shared with partition and Dowling lattices and resolving our Conjecture \ref{conj:shellable}.

\subsection{Topology}
As mentioned above, the poset $\D(G,S)$ 
arises naturally in the study of orbit configuration spaces, when we take $S$ to
be the set of singular points for the action of $G$ on $X$. Section \S \ref{sec:arrangements} is devoted to studying the topology of these spaces, and relating it to the combinatorics of Section \S\ref{sec:combinatorics}.

In \S\ref{sec:arrangementdefs}, we define an arrangement $\A(G,X)$ in $X^n$, whose complement is the orbit configuration space $\Conf_n^G(X\setminus S)$. Recall that the \emph{poset of layers} of an arrangement $\A(G,X)$ is the collection 
of connected components of intersections from $\A(G,X)$, ordered by reverse inclusion. 
This poset encodes subtle aspects of the topology of $\Conf_n^G(X\setminus S)$, 
as we shall see here (Theorem \ref{thm:layers}):
\begin{introthm}[\textbf{Poset of layers}]\label{thm:posetisom}
The poset of layers of the arrangement $\A(G,X)$ is naturally and
$\G$--equivariantly isomorphic to the poset $\D(G,S)$.	
\end{introthm}

This description opens the door to cohomology calculations: considering a
spectral sequence for complements of arrangements (see \cite{Petersen2017} and also
\cite{Totaro1996,bibby1,dupont}), one obtains a description of the $E_1$--page in
terms of the poset's Whitney homology.
Furthermore, when $X$ is a smooth projective algebraic variety, a
weight argument guarantees that there could be at most one nonzero
differential.
Thus, in this case one is closer to getting a hand on the cohomology.

We summarize the explicit description of the spectral sequence machinery, following the simplifications that arise from our combinatorial analysis, in 
the following (Theorem \ref{thm:ss}):

\begin{introthm}[\textbf{Simplified spectral sequence}]
\label{thm:spectral_sequence}
There is a spectral sequence with 
\[
E_1^{pq} = \bigoplus_{\beta\in\D^p(G,S)} \Ho_c^q(X^\beta)\otimes
\widetilde{\Ho}{}^{p-2}(\zero,\beta) \implies \Ho_c^{p+q}(\Conf_n^G(X\setminus
S)).
\]

Here, the summands are indexed by poset elements $\beta$ of rank $p$, and $X^\beta$ denotes the corresponding layer in $X^n$. The term
$\widetilde{\Ho}{}^{p-2}(\zero,\beta)$ denotes the reduced cohomology of the order
complex for the interval $(\zero,\beta)\subset \D(G,S)$, and is therefore described
explicitly by Theorem \ref{thm:intervals}.

When $X$ is a smooth projective variety, the sequence degenerates at the $E_2$--page, i.e. all differentials vanish past the first page.
\end{introthm}

Recall that certain invariants of $\Conf_n^G(X\setminus S)$ can be computed already from any page of a spectral sequence converging to $\Ho^*(\Conf_n^G(X\setminus S))$. These are the \emph{generalized Euler characteristics}, or \emph{cut-paste invariants}, discussed briefly in \S\ref{sec:motive}. Universal among those is the motive, i.e. the class $[\Conf_n^G(X\setminus S)]$ in the Grothendieck ring of varieties. Our combinatorial calculations then give:
\begin{introthm}[\textbf{Motivic factorization}]\label{thm:cut-paste}
Let $G$ act on an algebraic variety $X$ over an algebraically closed field $k$ as above, with singular set $S\neq \emptyset$. Then in the Grothendieck ring $\operatorname{K}_0(k)$,
\[
[\Conf_n^G(X\setminus S)] = \prod_{i=0}^{n-1} ([X]-|S|-|G|i).
\]
An analogous factorization for a free action is given in Theorem \ref{thm:motive} below.
\end{introthm}
In particular, this gives a formula for the number of $\mathbb{F}_q$-points in $\Conf_n^G(X\setminus S)$ for every $q$ divisible by $\operatorname{char}(k)$. Alternatively, when $X\setminus S$ is smooth, one gets a formula for the classical Euler characteristic
of $\Conf_n^G(X\setminus S)$.

In \S\ref{sec:localarr}, we analyze the local structure of the arrangement $\A(G,X)$, i.e. its germ at every point in $X^n$. A surprising conclusion is that, to first order, the arrangement $\A(G,X)$ is isomorphic to a product of orbit configuration spaces for groups possibly different from $G$. As a byproduct of our analysis we get a new proof of the following result.
\begin{introcor}[\textbf{Stabilizers on curves}]
Suppose $G$ acts faithfully on a algebraic curve $C$ over some algebraically closed field $k$. Then the stabilizer in $G$ of any smooth point is cyclic.
\end{introcor}

Lastly, our handle on the combinatorics of these arrangements can be exploited to understand what happens when one removes from $X$ a set $T$ other than the set of singular points $S$. We consider this more general case in \S\ref{sec:subarrangements}, but note now that all of our theorems hold true for these spaces as well.

For example, when $T$ is a $G$--invariant subset of $S$, the group $G$ now acts on $X\setminus T$ with nontrivial stabilizers. The resulting orbit configuration space
space $\Conf_n^G(X\setminus T)$ is the complement in $X^n$ of a subarrangement of $\A(G,X)$. The new poset of layers is a
subposet of $\D(G,S)$, which inherits many properties from $\D(G,S)$ to which our study applies. These types of arrangements arise naturally, e.g. from roots systems 
(see \S \ref{sec:subarrangements}).

\subsection{Acknowledgements}
An extended abstract of this work appeared in 
the proceedings of FPSAC 2018 \cite{FPSAC}.
The authors would like to thank Emanuele Delucchi, Graham Denham, and John Stembridge for many useful conversations and insights that helped shape this paper. 

\section{A generalization of Dowling lattices}\label{sec:combinatorics}

In \cite{Dowling1973}, Dowling defined a family of lattices $\D(G)$ dependent on
a positive integer $n$ and a finite group $G$. We recall his construction here.

Let $\n:=\{1,2,\dots,n\}$. 
A \textit{projectivized $G$-coloring} of a subset $B\subseteq\n$ is a function $b:B\to G$ defined up to the following equivalence:
$b:B\to G$ and $b':B\to G$ are \textbf{equivalent} if there is some $g\in G$ for which $b'=bg$.
The reader might benefit from thinking of such an equivalence class of colorings as a point in projective space
$$
[h_0:\ldots:h_d] \sim [h_0 g : \ldots : h_d g] \; \text{ for } g\in G.
$$
A \textbf{partial $G$-partition} of $\n$ is a set $\wt{\beta}=\{\wt{B_1},\dots,\wt{B_\ell}\}$ consisting of a partition $\beta=\{B_1,\dots,B_\ell\}$ of the subset $\cup B_i\subseteq\n$ along with projectivized $G$-colorings $b_i:B_i\to G$ for each $i$.
The \textbf{zero block} of a partial $G$-partition $\wt{\beta}$ of $\n$ 
is the set 
$Z:=\n\setminus\cup_{B\in\beta} B$.

\begin{notation}
We take the convention of using an uppercase letter $B$ for a set, the 
corresponding lowercase letter for the function $b:B\to G$, and $\wt{B}$ for
the equivalence class of $b:B\to G$.
\end{notation}

The \textbf{Dowling lattice} $\D(G)$ is the set of partial $G$--partitions of $\n$. 
We will consider the elements of $\D(G)$ as ordered pairs $(\wt{\beta},Z)$ where
$\wt{\beta}$ is a partial $G$--partition and $Z$ is its zero block. 
This set is a lattice with partial order 
determined by the following covering relations:
\begin{enumerate}
\item 
$(\wt{\beta}\cup\{\wt{A},\wt{B}\},Z)\prec (\wt{\beta}\cup\{\wt{C}\},Z)$
where $C=A\cup B$ with $c=a\cup bg$ for some $g\in G$, and
\item 
$(\wt{\beta}\cup\{\wt{B}\},Z)\prec(\wt{\beta},Z\cup B)$.
\end{enumerate}
The lattice $\D(G)$ has rank function given by
$\rk(\wt{\beta},Z)=n-\ell(\beta)$, where $\ell(\beta)$ is the number of blocks
in the partition $\beta$.
The (unique) minimum element of $\D(G)$ is the partition with each element in its own block and with no zero block, while the (unique) maximum element has every element in the zero block. See Figure \ref{fig:dowling lattice} for an example.

\begin{remark}
While it is not necessary to record the zero block in an element of $\D(G)$, we
do so because it is useful in understanding our generalization which 
involves adding a coloring to the zero block by some finite $G$-set.
\end{remark}

\subsection{Introducing the posets}\label{sec:posetdefs}

Let $G$ be a finite group acting on a finite set $S$.
\begin{definition}[\textbf{The $S$--Dowling poset}]\label{def:poset}
Let $\D(G,S)$ be the set of ordered pairs $\parts$ where $\wt{\beta}$ is
a partial $G$--partition of $\n$ and $z$ is an $S$--coloring of its zero block, 
i.e. a function $z:Z\rightarrow S$.
\end{definition}
\begin{notation}
To denote an element $(\wt{\beta},z)$ we will extend the standard notation of set partitions, as illustrated by the following example:
$$
[1_{g_1}3_{g_3}|2_{g_2}4_{g_4}6_{g_6}||5_{z_5}7_{z_7}]
$$
denotes the partial set partition $[13|246]$ with projectivized colorings $[g_1:g_3]$ and $[g_2:g_4:g_6]$ respectively, and zero block $\{5,7\}$ colored by the function $z$.
\end{notation}

The set $\D(G,S)$ is partially ordered with similar covering relations,
given by either merging two blocks or coloring one by $S$,
as follows:
\begin{description}
\item[(\textbf{merge})]
$(\wt{\beta}\cup\{\wt{A},\wt{B}\},z)\prec
(\wt{\beta}\cup\{\wt{C}\},z)$
where $C=A\cup B$ with $c=a\cup bg$ for some $g\in G$, and
\item[(\textbf{color})]
$(\wt{\beta}\cup\{\wt{B}\},z) \prec (\wt{\beta},z')$
where $z'$ is the extension of $z$ to 
$Z' = B\cup Z$ given on $B$ by a composition
\[
{B} \overset{b}{\rightarrow} G \overset{f}{\rightarrow} S
\]
for some $G$--equivariant function $f$.
\end{description}
Just as with the Dowling lattice, the poset $\D(G,S)$ is ranked with the
\textbf{rank} of $\parts$ given by $\rk\parts=n-\ell(\beta)$. 

\begin{remark}
When coloring a block of $\wt{\beta}$, the $G$--equivariant function $f:G\to S$
is determined by a choice of $f(e)=s\in S$, where $e$ is the identity in $G$. 
Then one can extend $z$ to $B$ by setting $z'(i)=b(i).s$ for $i\in B$. 
\end{remark}

Recall the wreath product $\G$, sometimes denoted by $G\wr\sym$, which is a semidirect
product of $G^n$ with the symmetric group $\sym$. It acts on the $S$--Dowling poset
 $\D(G,S)$ as follows.
Let $w=(g_1,\dots,g_n,\sigma)\in\G$ and $\parts\in\D(G,S)$. Then we have
$w.\parts=(\wt{\beta'},z')$ where 
\begin{itemize}
\item $\beta'=\{\sigma.B\ |\ B\in\beta\}$ with zero block
$\sigma.Z$,
\item $b'_i:(\sigma.B)\to G$ is given by $b'(\sigma(j))=g_jb(j)$, 
and
\item $z':(\sigma.Z)\to S$ is given by $z'(\sigma(j)) = g_j.z(j)$.
\end{itemize}
We leave it as an exercise to the reader to verify that the action preserves the
order. 

\begin{remark}\label{rmk:sets}
While it is convenient to consider (partial) partitions of the set
$\n=\{1,2,\dots,n\}$, it will sometimes prove to be more convenient to consider partitions of any finite set $\tau$, for example in Theorem
\ref{thm:intervals}.
That is, one could define $\D[\tau](G,S)$ as the set of partial $G$--partitions
of $\tau$ whose zero block is colored by $S$. 
In the case that $\tau=\n$, we have $\D[\n](G,S)=\D(G,S)$, and in general when
$|\tau|=n$ we have $\D[\tau](G,S)\cong\D(G,S)$. Note that the latter isomorphism depends on the choice of bijection $\tau\simeq \n$.
\end{remark}

\subsection{Examples}\label{sec:posetexs}

Here we introduce the primary examples, which will be carried throughout this
paper. The first describes the partition and Dowling lattices as specializations
of $S$--Dowling posets.

\begin{example}\label{ex:Dowling}
The Dowling lattice $\D(G)$ is equal to $\D(G,S)$ whenever $S$ consists of a
single point.
The Hasse diagram in the case $n=2$ and $G=\Z_2$ is depicted in Figure \ref{fig:Dowling}.

The lattice $\Q_{\n}$ of set partitions of $\n$ can be realized with the 
trivial group $G=\{1\}$ and no zero block, $\Q_{\n}\cong \D(\{1\},\emptyset)$.
As in \cite[Thm.~1(e)]{Dowling1973}, we also have
$\Q_{\n}\cong\D[n-1](\{1\})=\D[n-1](\{1\},\{0\})$. 

\begin{figure}[htb] \label{fig:dowling lattice}
\begin{tikzpicture}[scale=1]
\foreach \x in {-3,-1,1,3}
{
\draw[-] (\x,1.0)--(0,-1.0);
\draw[-] (\x,2.0)--(0,4.0);
}
\node at (0,-1.5)   {$[1_e|2_e||\emptyset]$};
\node at (-3,1.5)   {$[2_e||1_0]$};
\node at (-1,1.5)   {$[1_e||2_0]$};
\node at (1,1.5)   {$[1_e2_e||\emptyset]$};
\node at (3,1.5)    {$[1_e2_\iota||\emptyset]$};
\node at (0,4.5) {$[\emptyset||1_02_0]$};
\end{tikzpicture}
\caption{Type C Dowling lattice $\D[2](\Z_2)=\D[2](\Z_2,\{0\})$. See Example \ref{ex:Dowling}.}
\label{fig:Dowling}
\end{figure}
\end{example}

\begin{example}[\textbf{Type C Dowling poset}]\label{ex:typeC}
Let $G=\Z_2$ act trivially on a finite set $S$.
In the case that $|S|$ is 2 or 4, the poset $\D(G,S)$ was studied in
\cite{bibby2}.
Here, the poset describes the combinatorial structure of an arrangement arising
naturally from the type C root system, which we will revisit in Example
\ref{ex:typeCP}.

The Hasse diagram for $\D[2](G,S)$ when $S=\{\pm1\}$ is depicted in Figure \ref{fig:DZ2-toric}.

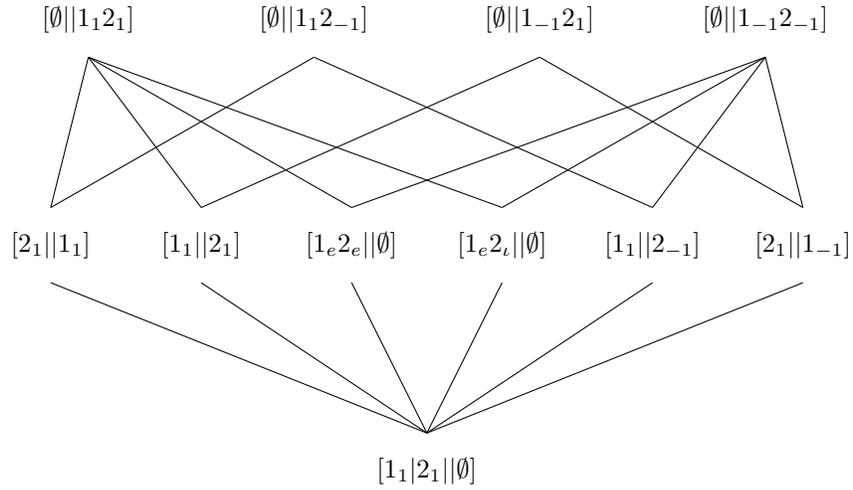
\begin{figure}[htb]
\begin{tikzpicture}[scale=1]
\foreach \x in {-5,-3,-1,1,3,5}
{
\draw[-] (\x,1.0)--(0,-1.0);
}
\foreach \x in {-5.0,-3.0,-1,1}
{
\draw[-] (-4.5,4.0)--(\x,2.0);
\draw[-] (4.5,4.0)--(-\x,2.0);
}
\draw[-] (1.5,4.0)--(5,2.0);
\draw[-] (-1.5,4.0)--(3,2.0);
\draw[-] (1.5,4.0)--(-3,2.0);
\draw[-] (-1.5,4.0)--(-5,2.0);
\node at (0,-1.5)   {$[1_e|2_e||\emptyset]$};
\node at (-5,1.5)   {$[2_e||1_{1}]$};
\node at (-3,1.5)   {$[1_e||2_{1}]$};
\node at (-1,1.5)   {$[1_e2_e||\emptyset]$};
\node at (1,1.5)    {$[1_e2_\iota||\emptyset]$};
\node at (3,1.5)    {$[1_e||2_{-1}]$};
\node at (5,1.5)    {$[2_e||1_{-1}]$};
\node at (-4.5,4.5) {$[\emptyset||1_12_1]$};
\node at (-1.5,4.5) {$[\emptyset||1_12_{-1}]$};
\node at (1.5,4.5)  {$[\emptyset||1_{-1}2_1]$};
\node at (4.5,4.5)  {$[\emptyset||1_{-1}2_{-1}]$};
\end{tikzpicture}

\caption{Type C Dowling poset $\D[2](\Z_2,\{\pm1\})$. See Examples
\ref{ex:typeC} and \ref{ex:typeCP}, and note the isomorphism with the poset
depicted in Figure \ref{fig:PZ2-toric}. The quotient of this poset by the action
of $\sym[2][\Z_2]$ is depicted in Figure \ref{fig:OZ2}.}
\label{fig:DZ2-toric}
\end{figure}

\end{example}

\begin{example}\label{ex:Z2}
In contrast to the last example, with $\Z_2$ acting trivially on $\{\pm1\}$,
consider the nontrivial action of $\Z_2$ on $\{\pm1\}$. These two $S$--Dowling
posets have the same underlying set but a different partial order. 
In fact, these two posets are not even isomorphic: in one example every maximal element covers 3 elements, while in the other maximal elements cover either 2 or 4. 
The Hasse
diagrams for the trivial and nontrivial actions are depicted in Figures
\ref{fig:DZ2-toric} and \ref{fig:Z2}.

\begin{figure}[htb]
\begin{tikzpicture}[scale=1]
\foreach \x in {-5,-3,-1,1,3,5}
{
\draw[-] (\x,1.0)--(0,-1.0);
}
\foreach \x in {-5.0,-3.0}
{
\draw[-] (-4.5,4.0)--(\x,2.0);
\draw[-] (4.5,4.0)--(-\x,2.0);
}
\draw[-] (-1.5,4)--(1,2)--(1.5,4);
\draw[-] (-4.5,4)--(-1,2)--(4.5,4);
\draw[-] (1.5,4.0)--(5,2.0);
\draw[-] (-1.5,4.0)--(3,2.0);
\draw[-] (1.5,4.0)--(-3,2.0);
\draw[-] (-1.5,4.0)--(-5,2.0);
\node at (0,-1.5)   {$[1_e|2_e||\emptyset]$};
\node at (-5,1.5)   {$[2_e||1_{1}]$};
\node at (-3,1.5)   {$[1_e||2_{1}]$};
\node at (-1,1.5)   {$[1_e2_e||\emptyset]$};
\node at (1,1.5)    {$[1_e2_\iota||\emptyset]$};
\node at (3,1.5)    {$[1_1||2_{-1}]$};
\node at (5,1.5)    {$[2_1||1_{-1}]$};
\node at (-4.5,4.5) {$[\emptyset||1_12_1]$};
\node at (-1.5,4.5) {$[\emptyset||1_12_{-1}]$};
\node at (1.5,4.5)  {$[\emptyset||1_{-1}2_1]$};
\node at (4.5,4.5)  {$[\emptyset||1_{-1}2_{-1}]$};
\end{tikzpicture}
\caption{The Dowling poset $\D[2](\Z_2,\{\pm1\})$ where $\Z_2$ acts
nontrivially on $\{\pm1\}$ (see Example \ref{ex:Z2}). 
Compare to Figure \ref{fig:DZ2-toric}, where there
is an obvious set bijection that does not preserve the order.}
\label{fig:Z2}
\end{figure}
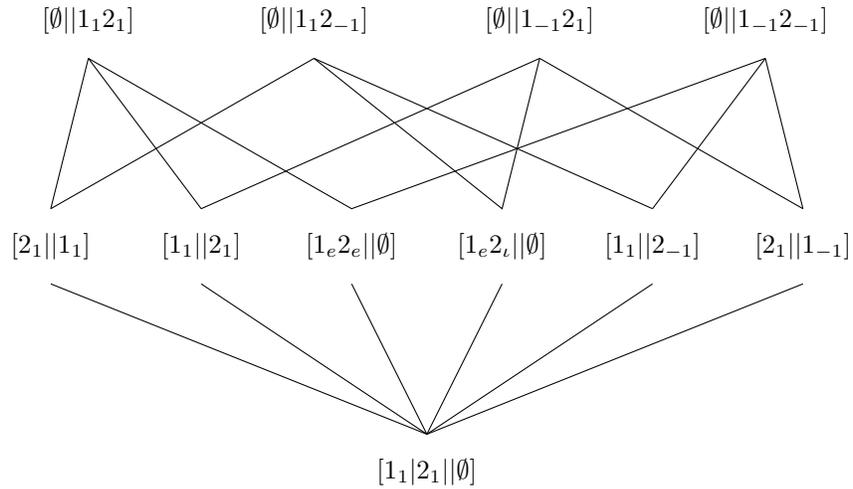
\end{example}

\begin{example}[\textbf{Hexagonal Dowling poset}]\label{ex:Z6}
Let $G=\Z_6$, which we identify with the group of $6$th roots of unity
$\{\pm1,\pm\zeta_3,\pm\zeta_3^2\}$. Let us consider $G$ acting on the set 
$S=\{e, z_1, z_2, z_3, w_1, w_2\}$ so that the action of the generator $-\zeta_3$
is by the permutation given by cycle notation $(e)(z_1,z_2,z_3)(w_1,w_2)$.
This poset also arises from an arrangement, which we will revisit in Example
\ref{ex:Z62} below.

Even when $n=2$ this poset is large: There are $6^2=36$ maximal elements, all
with rank two, corresponding to the possible $S$--colorings of $\{1,2\}$.
We include in Figure \ref{fig:OZ6-OZ4} the Hasse diagram for the orbits 
of $\D[2](G,S)$ under the $\G[2]$-action. We will revisit this orbit space in 
Example \ref{ex:Z6orbits} below.
\end{example}

\begin{example}[\textbf{Square Dowling poset}]\label{ex:Z4}
Let $G=\Z_4=\{\pm1,\pm i\}$ act on the set $S=\{e,z_1,z_2,t\}$,
 where the action of $i$ on $S$ is by the permutation $(e)(z_1,z_2)(t)$.
As with the previous examples, this is associated to an arrangement which we
revisit in Example \ref{ex:Z42}.

We later depict a subposet of $\D[2](\Z_4,S)$ in Figure 
\ref{fig:interval}, the intervals under $[\emptyset||1_{z_1}2_{z_2}]$ and
$[\emptyset||1_t2_t]$. It is interesting to note that these intervals are isomorphic to 
$\D[2](\Z_2)$ and $\D[2](\Z_4)$, respectively, which will turn out to be a general 
phenomenon (see Theorem \ref{thm:intervals} below).
\end{example}

\subsection{Functoriality} \label{sec:functoriality}
The first property we state for the posets $\D(G,S)$ is that they behave
naturally with respect to changes in the inputs of $n$, $G$, and $S$. 
This proposition is straightforward to verify.

\begin{proposition}[\textbf{Functoriality}]\label{prop:functoriality}
The $S$--Dowling posets are functorial in the following ways:
\begin{enumerate}

\item \label{prop:monoidality}
The two inclusions $\n, \n[k]\into\n\sqcup\n[k]$ induce a 
$(\G\times\G[k])$--equivariant injective map of ranked posets 
$\D(G,S)\times\D[k](G,S)\to\D[n+k](G,S)$ defined by
\[((\wt{\beta},z),(\wt{\beta}',z')) \mapsto (\wt{\beta}\cup\wt{\beta}',z\cup
z').\]
Here the sets $\n$ and $\n[k]$ could be replaced by any two finite sets.

In particular, multiplying by $\zero\in \D[k](G,S)$ gives a canonical equivariant 
injection $\D(G,S)\to \D[n+k](G,S)$ and exhibits $\D[\bullet](G,S)$ as a functor, 
as stated next.

\item \label{prop:changen} If $\iota:\n\to\n[m]$ is an injective map of sets, 
then there is an injective map of ranked posets $\iota_*:\D(G,S)\to\D[m](G,S)$
defined by 
\[\iota_*(\wt{\beta},z) = (\wt{\beta}'\cup\{\{j\}\ |\ j\notin\im\iota\},z)\]
where $\wt{\beta}'$ is defined via $b\circ \iota^{-1}:\iota(B)\to G$ for each
$\wt{B}\in\wt{\beta}$.

\item \label{prop:changeS}
If $f:S\to T$ is a $G$--equivariant map of sets, then there is a 
$\G$--equivariant map of ranked posets $f_*:\D(G,S)\to\D(G,T)$ defined by
\[f_*(\wt{\beta},z) = (\wt{\beta},f\circ z).\]
Furthermore, if $f$ is surjective (resp. injective) then 
$f_*$ is also surjective (resp. injective).

\item \label{prop:changeG}
If $\nu:G\to H$ is a group homomorphism and $H$ acts on a set $S$ (hence
also inducing an action of $G$ on $S$),
then there is a $\G$--equivariant map of ranked posets $\nu_*:\D(G,S)\to\D(H,S)$
defined by \[\nu_*(\wt{\beta},z) = (\wt{\beta}',z)\]
where $\wt{\beta}'$ is defined via $\nu\circ b:B\to H$ for each
$\wt{B}\in\wt{\beta}$.
Furthermore, if $\nu$ is surjective (resp. injective) then 
$\nu_*$ is also surjective (resp. injective).
\end{enumerate}
\end{proposition}

\begin{remark}
\begin{enumerate}
\item The maps defined in 
Proposition \ref{prop:functoriality}\eqref{prop:monoidality} and \eqref{prop:changen} will be most useful in
studying representation stability in our sequel to this paper \cite{BG2}, when we consider the
sequence of posets $\D(G,S)$ as $n$ grows.
One can view this map $\iota$ as a stabilization operation, padding our partial partitions with singleton
blocks.

\item Rephrasing Proposition \ref{prop:functoriality}\eqref{prop:monoidality} in view
of Remark \ref{rmk:sets} and applying it to general finite sets, it follows that $\D[\bullet](G,S)$ is a \emph{monoidal} functor from 
the category
$(\mathtt{FI},\coprod)$ of finite sets and injections, to
the category \texttt{Pos} of posets.

\item A consequence of Proposition \ref{prop:functoriality}\eqref{prop:changeS} is
that if $|S|>1$, the Dowling lattice $\D(G)$ is a quotient of the poset
$\D(G,S)$, where the fiber above an element $\wt{\beta}\in\D(G)$ consists of all the
possible $S$-colorings of its zero block.
\end{enumerate}
\end{remark}

\subsection{Local structure of the poset}\label{sec:posetprops}

In this section, we give an explicit description of the intervals inside
of the $S$--Dowling posets. The beauty of the local structure (the intervals)
is that we can view our posets as being built out of partition and Dowling
lattices, which we denote by $\Q_{\n}$ and $\D(G)$.

\begin{remark}\label{rmk:notalattice}
While the local structure of the poset is familiar, it is important to note that
the global structure is more complicated.
When $|S|>1$, the poset $\D(G,S)$ is not even a (semi)lattice: 
least upper bounds and greatest lower bounds need not exist.
This is demonstrated in Example \ref{ex:Z2} (see Figure \ref{fig:DZ2-toric}): there exist pairs of elements with multiple minimal upper bounds, and other pairs with no upper bound.

Furthermore, while $\D(G,S)$ has a (unique) minimum element $\zero$ given by the
partition of $\n$ into $n$ singleton blocks, it may have several maximal elements corresponding to different $S$--colorings of $\n$. 
\end{remark}

The following theorem describes the intervals in $\D(G,S)$. It is a generalization
of Dowling's \cite[Theorem 2 and Corollary 2.1]{Dowling1973}, which are the
case with $|S|=1$. The surprising consequence of Theorem \ref{thm:intervals} below, is that when $|S|\geq 2$ the local picture is sensitive to the orbits and stabilizers of $S$.
As mentioned in Remark \ref{rmk:sets}, this theorem is most naturally stated by
considering $\D[\tau](G,S)$ for a general finite set $\tau$ (not just $\n$); we discuss
this further in Remark \ref{rmk:intervals} below.

\begin{theorem}[\textbf{Local structure}]\label{thm:intervals}
Let $S$ be a finite set with an action of a finite group $G$,
and let $\O(S)$ denote its set of $G$--orbits.
For each orbit $\oo\in\O(S)$, pick a representative $s_\oo\in\oo$ and let $G_\oo$ be
the stabilizer of $s_\oo$ in $G$.

For $(\wt{\beta},z_\beta)\in\D(G,S)$, we have 
\begin{equation}\label{eq:lowerint}
\D(G,S)_{\leq(\wt{\beta},z_\beta)}\cong 
\prod_{B\in\beta} \Q_{B} \times \prod_{\oo\in\O(S)} \D[z_{\beta}^{-1}(\oo)](G_\oo)
\end{equation}
and
\begin{equation}\label{eq:upperint}
\D(G,S)_{\geq(\wt{\beta},z_\beta)} \cong \D[\beta](G,S).
\end{equation}
Furthermore, if $(\wt{\alpha},z_\alpha)\leq(\wt{\beta},z_\beta)$, then 
\begin{equation}\label{eq:interval}
[(\wt{\alpha},z_\alpha),(\wt{\beta},z_\beta)]
\cong \prod_{B\in\beta} \Q_{p_B} \times \prod_{\oo\in\O(S)}
\D[r_\oo](G_\oo),
\end{equation}
where $p_B$ is the set of blocks $\wt{A}\in\wt{\alpha}$ for which 
$A\subseteq B$ and $r_\oo$ is the set of blocks 
$\wt{A}\in\wt{\alpha}$ for which $A\subseteq z_{\beta}^{-1}(\oo)$.

In particular, every closed interval is a geometric lattice.
\end{theorem}

\begin{remark}\label{rmk:intervals}
\begin{enumerate}
\item 
One can write the product decompositions above more explicitly by denoting
$\beta=\{B_1,\dots,B_\ell\}$, $\O(S)=\{\oo_1,\dots,\oo_k\}$, $n_i=|B_i|$, and
$m_j=|z^{-1}(\oo_j)|$. Then Theorem
\ref{thm:intervals}\eqref{eq:lowerint} says:
\[[\zero,(\wt{\beta},z_\beta)] \cong \Q_{\n[n_1]}\times\cdots\times
\Q_{\n[n_\ell]}\times\D[m_1](G_{\oo_1})\times\cdots\times\D[m_k](G_{\oo_k}).\]
\item 
Note that in Theorem \ref{thm:intervals}\eqref{eq:upperint} and
\eqref{eq:interval}, the base sets of the lattices are sets of \emph{blocks}
rather than subsets of $\n$. 
\end{enumerate}
\end{remark}

We will prove parts \eqref{eq:lowerint} and \eqref{eq:upperint} of Theorem
\ref{thm:intervals} below; part \eqref{eq:interval} follows by combining these two.
To qualitatively explain part \eqref{eq:lowerint}, 
recall that the covering relations state that an
element $(\wt{\alpha},z_\alpha)$ lies under $(\wt{\beta},z_\beta)$ if the
partition $\alpha$ is a refinement of the partition $\beta$, possibly excising
blocks away from the zero block $Z_\beta$. In particular, $\alpha$ defines a
partition of each block in $\beta$, and furthermore includes a partial partition
of the zero block. 
To qualitatively explain part \eqref{eq:upperint},
recall that the covering relation allows one to merge existing blocks, or throw
entire blocks into the zero block. Thus an element above $(\wt{\beta},z_\beta)$
is determined by specifying blocks to be merged and the $G$--ratios between their
$G$--colorings, and by coloring the remaining blocks by $S$. 

\begin{proof}[Proof of Theorem \ref{thm:intervals}\eqref{eq:lowerint}]
The isomorphism is not canonical; we make the following choices.
For each $\oo\in\O(S)$, fix a representative $s_\oo\in\oo$ with stabilizer
subgroup $G_\oo$. Then for each element $t\in \oo$ pick a `transporter' $g_t\in G$ so that 
$g_t.s_\oo=t$.

We define a map from $[\zero,(\wt{\beta},z_\beta)]$ to the product by
an assignment \[(\wt{\alpha},z_\alpha)\mapsto
((\alpha_B)_{B\in\beta},(\wt{\alpha}_\oo)_{\oo\in\O(S)}).\]
The first tuple is simple: for a block $B\in \beta$, define a partition of $B$ by
\[\alpha_B:=\{A\in\alpha\ |\ A\subseteq B\}.\]
For the second tuple start with defining for every
$\oo\in\O(S)$ a partial partition of $z_\beta^{-1}(\oo)$ by
\[\alpha_\oo:=\{A\in\alpha\ |\ A\subseteq z_\beta^{-1}(\oo)\}.\]
For each $A\in\alpha_\oo$, the \emph{coloring} relation in the definition of
$\prec$ shows that we may pick a representative $a:A\to G$ such that
$z_\beta(i)=a(i).s_\oo$ for each $i\in A$. 
Since $z_\beta(i)=g_{z_\beta(i)}.s_\oo$, we can define a $G_\oo$--coloring  $a':A\to G_\oo$ with $a'(i)=g^{-1}_{z_\beta(i)}a(i)$.

To describe the inverse map, consider a pair of tuples
$((\alpha_B)_{B\in\beta},(\wt{\alpha}_\oo)_{\oo\in\O(S)})$. 
We recover a partial partition of $\n$ by unioning all of these (partial)
partitions,
\[\alpha:= \{A\in \alpha_B\ |\ B\in\beta\}\cup\{A\in\alpha_\oo\ |\
\oo\in\O(S)\},\]
and obtain $G$-colorings $a:A\to G$ as follows. If $A\in \alpha_B$, then inherit $a=b|_A:A\to G$ from $\wt{B}\in\wt{\beta}$. Otherwise $A\in \alpha_\oo$ is a block of a partial partition. Represent its $G_\oo$-coloring by $a':A\to G_\oo$ and recolor by $a(i)=g_{z_\beta(i)} a'(i)$. Lastly, inherit the $S$-coloring of the zero block $Z_\alpha$ from that of $Z_\beta$. These constructions are clearly inverses.
\end{proof}

\begin{proof}[Proof of Theorem \ref{thm:intervals}\eqref{eq:upperint}]
 This isomorphism is also not canonical:  Let us write 
$\wt{\beta}=\{\wt{B}_1,\dots,\wt{B}_t\}$ and pick a representative
$b_i:B_i\to G$ for each $\wt{B}_i\in\wt{\beta}$.

Then for $(\wt{\alpha},z_\alpha)\geq(\wt{\beta},z_\beta)$ 
we will construct an element of
$\D[t](G,S)$. For $\wt{A}\in\wt{\alpha}$, let
\[C_A:=\{i\in\n[t]\ |\ B_i\subseteq A\}.\]
Now, pick a representative $a:A\to G$ for $\wt{A}\in\wt{\alpha}$. By the
covering relations, we have that for each $i\in C_A$ there is $g_i\in G$ 
such that $a=b_ig_i$. This defines a function $c_A:C_A\to G$; a different
representative map for $\wt{A}$ would define a function in the same equivalence
class of $c_A$. The collection $\{\wt{C}_A\ |\ \wt{A}\in\wt{\alpha}\}$ is a partial
$G$-partition of the set $\n[t]$. 

We also have that $Z_\alpha=Z_\beta\cup Z$ where $Z$ is a union of blocks from
$\beta$ and $z_\alpha|_{Z_\beta}=z_\beta$. 
In fact, the zero block of $\{\wt{C}_A\}$ is $Z'=\{i\in\n[t]\ |\
B_i\subseteq Z\}$. Since $z_\alpha|_{B_i}=f_i\circ b_i$ for some 
(unique) $G$-equivariant 
$f_i:G\to S$, let us color $Z'$ so that $z'(i)=f_i(e)$. Then $(\{\wt{C}_A\ |\
\wt{A}\in\wt{\alpha}\},z')\in \D[t](G,S)$, and one can recover
$(\wt{\alpha},z_\alpha)$ from this data.
\end{proof}

\begin{example}[\textbf{Hexagonal Dowling Poset}]\label{ex:intervalZ6}
Recall from Example \ref{ex:Z6}
the Hexagonal Dowling poset $\D(\Z_6,S)$, where
$S=\{e,z_1,z_2,z_3,w_1,w_2\}$. 
The orbits of $S$ are $\{e\}$, $\{z_1,z_2,z_3\}$, $\{w_1,w_2\}$, so that we have
$\O(S)=\{\oo(e),\oo(z),\oo(w)\}$ with stabilizers $G_e=\Z_6$,
$G_{z_1}=\{\pm1\}\cong\Z_2$, and $G_{w_1}=\{1,\zeta,\zeta^2\}\cong\Z_3$.

One already finds intervals that factor as Dowling lattices of different ranks and of different groups when $n=5$: consider $\parts\in \D[5](\Z_6,S)$ given by
\[
\parts=[\emptyset|| 1_{z_1}2_{z_1}3_{z_2}4_{w_1}5_{w_1}]
\]
then
\begin{align*}
[\zero,\parts] &\cong
 \D[{\{1,2,3\}}](\Z_2)\times
\D[{\{4,5\}}](\Z_3)\\
&\cong\D[3](\Z_2)\times\D[2](\Z_3).
\end{align*}

Next, the element $(\wt{\alpha},z_\alpha)\in[\zero,\parts]$ given by $(\wt{\alpha},z_\alpha) = [ 1_{1}3_{\zeta} || 2_{z_1}4_{w_1}5_{w_1} ]$ is mapped under the isomorphism to
\[
( [1_1 3_1 || 2], [\emptyset || 4,5] ) \in \D[3](\Z_2)\times\D[2](\Z_3).
\]
Moreover, $\D[5](\Z_6,S)_{\geq (\alpha,z_\alpha)} \cong \D[1](G,S)$ and the element $\parts$ is mapped under this isomorphism to $[\emptyset|| 1_{z_1}]$.
\end{example}

\begin{example}[\textbf{Square Dowling poset}]\label{ex:intervalZ4}
Recall from Example \ref{ex:Z4} the square Dowling poset $\D(\Z_4,S)$ where 
$S=\{e,z_1,z_2,t\}$. The orbits of $S$ are $\{e\}$, $\{z_1,z_2\}$, and $\{t\}$, 
so that we have $\O(S) = \{\oo(e),\oo(z),\oo(t)\}$ with stabilizers 
$G_e=G_t=\Z_4$ and $G_{z_1}=\{\pm1\}\cong\Z_2$.

Figure \ref{fig:interval} depicts two intervals inside $\D[2](\Z_4,S)$, where 
\[[\zero,[\emptyset||1_{z_1}2_{z_2}]]\cong \D[2](\Z_2)
\ \ \text{and}\ \ 
[\zero,[\emptyset||1_t2_t]]\cong \D[2](\Z_4)\]

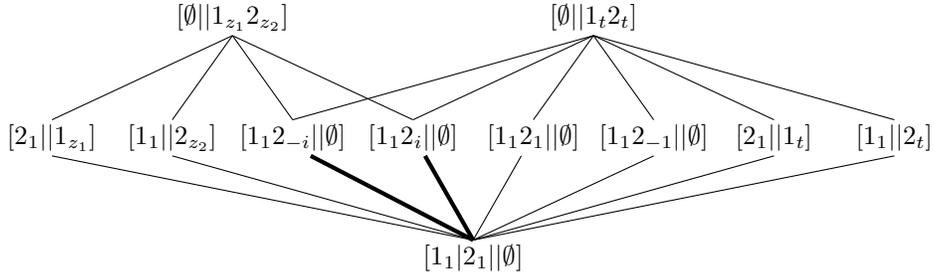
\begin{figure}[htb]
\begin{tikzpicture}[scale=0.8]
\node  at (0,0)  {$[1_1|2_1||\emptyset]$};
\node  at (-7,2) {$[2_1||1_{z_1}]$};
\node  at (-5,2) {$[1_1||2_{z_2}]$};
\node  at (-3,2) {$[1_12_{-i}||\emptyset]$};
\node  at (-1,2) {$[1_12_{i}||\emptyset]$};
\node  at (1,2)  {$[1_12_{1}||\emptyset]$};
\node  at (3,2)  {$[1_12_{-1}||\emptyset]$};
\node  at (5,2)  {$[2_1||1_{t}]$};
\node  at (7,2)  {$[1_1||2_{t}]$};
\node  at (-4,4) {$[\emptyset||1_{z_1}2_{z_2}]$};
\node  at (2,4)  {$[\emptyset||1_{t}2_{t}]$};
\draw[-]
(0,0.3) edge (-7,1.7)
(0,0.3) edge (-5,1.7)
(0,0.3) edge[ultra thick] (-2.7,1.7)
(0,0.3) edge[ultra thick] (-0.8,1.7)
(0,0.3) edge (0.8,1.7)
(0,0.3) edge (3,1.7)
(0,0.3) edge (5,1.7)
(0,0.3) edge (7,1.7)
(-4,3.7) edge (-7,2.3)
(-4,3.7) edge (-5,2.3)
(-4,3.7) edge (-3,2.3)
(-4,3.7) edge (-1,2.3)
(2,3.7) edge (-3,2.3)
(2,3.7) edge (-1,2.3)
(2,3.7) edge (1,2.3)
(2,3.7) edge (3,2.3)
(2,3.7) edge (5,2.3)
(2,3.7) edge (7,2.3);
\end{tikzpicture}
\caption{Two overlapping intervals inside the square Dowling poset
$\D[2](G,S)$ with $G=\Z_4$ acting on $S=\{e,z_1,z_2,t\}$ as in
Example \ref{ex:Z4}. See Example \ref{ex:intervalZ4}.}
\label{fig:interval}
\end{figure}
\end{example}

\subsection{Characteristic polynomial}\label{sec:charpoly}

A fundamental invariant attached to a ranked poset is its \textbf{characteristic polynomial}.
Recall (and see \cite{OS1980}) that when specialized to intersection lattices of hyperplane 
arrangements this polynomial gives a close relative of the Poincar\'{e} polynomial of the 
complement, and that in the further special case of graphical arrangements it computes the 
chromatic polynomial of the graph.
The roots of this polynomial carry subtle information, e.g. for reflection arrangements it 
encodes the exponents of the Coxeter group (see \cite{Brieskorn1973}). In Theorem 
\ref{thm:charpoly} below we factorize the polynomial associated with $\D(G,S)$.

The factorization
formula in Theorem \ref{thm:charpoly} specializes to that
computed by Dowling \cite[Thm.~5]{Dowling1973} when $|S|=1$ and by Ardila,
Castillo, and Henley \cite[Thm.~1.18]{ACH2015} when $G=\Z_2$ and $|S|=2$.
The same formula for the partition lattice is well-known and goes back to Arnol'd 
\cite{Arnold1969} and Stanley \cite{Stanley1972}. This is the special case when
$G$ is trivial and $S$ is empty.

Recall that the characteristic polynomial of a ranked poset $P$ with minimum
element $\zero$ is defined by
\[\chi(P;t) = \sum_{x\in P}\mu(\zero,x)t^{\rk(P)-\rk(x)},\]
where $\mu$ is the M\"obius function. In \cite{Stanley1972} Stanley defined the notion of a supersolvable lattice, encompassing the cases of partition and Dowling lattices.
There he showed that for such lattices, a partition of the atoms gives a factorization of their characteristic polynomial.
We thus proceed by describing the atoms, generalizing Corollaries 1.1 and 1.2 of \cite{Dowling1973}.

\begin{lemma}\label{lem:atoms}
The rank-one elements (or atoms) of $\D(G,S)$ are 
\begin{enumerate}
\item 
$\alpha_{ij}(g):=[i_1 j_g|1|2|\ldots|\hat{i}|\ldots|\hat{j}|\ldots|n||\emptyset]$ where $1\leq i<j\leq n$ and $g\in G$, corresponding to 
the  
$G$-partition
whose only non-singleton block is 
$\{i,j\}$ with
$G$-coloring $[1:g]$; and
\item $\alpha_i^s:=[1|\ldots|\hat{i}|\ldots|n||i_{s}]$, where $1\leq i\leq n$ and $s\in S$, corresponding to
the partial $G$-partition with zero block $\{i\}$ colored by $s$ and the rest are singleton blocks.
\end{enumerate}

Moreover, if $\parts\in\D(G,S)$ has rank $n-\ell$, then it is covered by
\[ \ell|S| + \binom{\ell}{2}|G| \]
elements of rank $n-\ell+1$. 
\end{lemma}
\begin{proof}
The description of atoms follows directly from the covering relations. 
Thus, the poset has $n|S|+\binom{n}{2}|G|$ atoms. To count the number of
elements covering some $(\wt{\beta},z)$ of rank $n-\ell$, 
we recognize them as atoms of
$\D[\ell](G,S)$ via Theorem \ref{thm:intervals}\eqref{eq:upperint}.
\end{proof}

\begin{theorem}[\textbf{Characteristic polynomial}]\label{thm:charpoly}
If $S$ is a nonempty finite set, then 
\[\chi(\D(G,S);t) = \prod_{i=0}^{n-1} (t-|S|-|G|i). \]

If $S=\emptyset$, then
\[\chi(\D(G,\emptyset);t) = \prod_{i=1}^{n-1} (t-|G|i). \]
\end{theorem}

\begin{proof}
The authors thank Emanuele Delucchi for suggesting this method of proof.

First assume that $S$ is nonempty.
For each $k\in\n$, define
\[A_k:=\{\alpha_k^s\ |\ s\in S\} \cup \{\alpha_{ik}(g)\ |\ g\in G, i< k\}.\]
The sets $A_1,\dots,A_n$ give a partition of the atoms of $\D(G,S)$. Further
define $\hat{A}_{k}:=A_{k}\cup\{\zero\}$, a subposet of $\D(G,S)$.

By choice of the partition, least upper bounds of elements from distinct blocks
exist, and so there is a well-defined map 
$f:\hat{A}_1\times\cdots\times \hat{A}_n\to \D(G,S)$ 
defined by $f(x_1,\dots,x_n)=x_1\vee\cdots \vee x_n$.
It is easy to verify that this is a complete transversal function in the sense
of \cite[Def.~4.2]{Hallam2014} and that the following properties hold:
\begin{enumerate}
\item For each $(x_1,\dots,x_n)\in \hat{A}_1\times\cdots\times\hat{A}_n$, the
number of $i$ for which $x_i\neq\zero$ is equal to the rank of
$x_1\vee\cdots\vee x_n$.
\item For each $(\wt{\beta},z)\in\D(G,S)$, if $k$ is the minimum such that
$A_k\cap[\zero,(\wt{\beta},z)]\neq\emptyset$, then 
$|A_k\cap[\zero,(\wt{\beta},z)]|=1$.
\end{enumerate}
The formula then follows from Theorem 4.4 of \cite{Hallam2014}, with the roots
of the polynomial given by 
$|A_{k}|=|S|+|G|(k-1)$ for $k=1,2,\dots,n$.
Finally, if $S=\emptyset$, then the sets $A_2,\dots,A_n$ give the necessary
partition and roots $|G|(k-1)$ for $k=2,\dots,n$. 
\end{proof}

\begin{remark}
It is interesting to note that the characteristic polynomial does not depend on
the action of $G$ on $S$, but rather only on $n$, $|G|$, and $|S|$. For
example, the posets $\D[2](\Z_2,\{\pm1\})$ where $\Z_2$ acts either trivially
or nontrivially on $\{\pm1\}$, depicted in Figures \ref{fig:DZ2-toric} and
\ref{fig:Z2}, have the same characteristic polynomial.
\end{remark}

\begin{example}[\textbf{Hexagonal Dowling poset}]\label{ex:Z6charpoly}
If $G=\Z_6$ and $|S|=6$ (eg. in Example \ref{ex:Z6}),
the first few characteristic polynomials are
\begin{align*}
\chi(\D[2](G,S);t) &= (t-6)(t-12)\\ &= t^2-18t+72\\
\chi(\D[3](G,S);t) &= (t-6)(t-12)(t-18)\\ &= t^3-36t^2+396t-1296\\
\chi(\D[4](G,S);t) &= (t-6)(t-12)(t-18)(t-24)\\ &=
t^4-60t^3+1260t^2-10800t+31104
\end{align*}
\end{example}

\begin{example}[\textbf{Square Dowling poset}]\label{ex:Z4charpoly}
If $G=\Z_4$ and $|S|=4$ (eg. in Example \ref{ex:Z4}), 
the first few characteristic polynomials are 
\begin{align*}
\chi(\D[2](G,S);t) &= (t-4)(t-8)\\ &= t^2-12t+32\\
\chi(\D[3](G,S);t) &= (t-4)(t-8)(t-12)\\ &= t^3-24t^2+176t-384\\
\chi(\D[4](G,S);t) &= (t-4)(t-8)(t-12)(t-16)\\ &= t^4-40t^3+560t^2-3200t+6144
\end{align*}
\end{example}

\subsection{Orbits and labeled partitions}\label{sec:orbits}

A set partition of $\n$ determines a partition of the number $n$ by taking the size of its
blocks. The action
of $S_n$ on the partition lattice $\Q_{\n}$ preserves this data, and in fact the
$S_n$-orbits of $\Q_{\n}$ are in bijection with partitions of $n$. 
The action of $\G$ on the Dowling lattice $\D(G)$ behaves similarly, 
except that the zero block cannot be permuted among the other blocks. 
The orbit can then be recorded by a partition of $n$
with a distinguished part denoting the size of the zero block.

Let $\parts\in\D(G,S)$, and let $\O:=\O(S)$. 
The action of $\G$ preserves the block sizes in the partial $G$-partition
$\beta$, and cannot permute the sets $z^{-1}(\oo)$ amongst each other or with
blocks in $\beta$. In particular, the sizes of the sets $z^{-1}(\oo)$ are also preserved. 
This data can be recorded by a partition of $n$, where some parts are
labeled by elements of $\O$, recording the sizes of $z^{-1}(\oo)$ for $\oo\in \O$. This notion is made precise by the following definition.

\begin{definition}[\textbf{Labeled partitions}]\label{def:partitions}
An \textbf{$\O$-labeled partition of $n$} is an integer partition of $n$, i.e. a collection of positive integers summing to $n$, with some parts
colored by $\O$ so that each color is used at most once.
For $\O=\{\oo_1,\dots,\oo_s\}$ and an $\O$-labeled 
partition $\lambda$, we will use the notation 
\[\lambda=(\lambda_1,\dots,\lambda_\ell||\lambda_{\oo_1},\dots,\lambda_{\oo_s})\]
where $\lambda_1\geq\cdots\geq\lambda_\ell>0$ are the uncolored (or unlabeled) parts,
$\lambda_\O=(\lambda_{\oo_1},\dots,\lambda_{\oo_s})$ is the
colored (or labeled) portion and $\sum\lambda_i+\sum\lambda_\oo = n$. We typically omit the labeled parts which are zero.

Let $\Q_n(\O)$ denote the set of all $\O$-labeled partitions of $n$.
\end{definition}

The following theorem exhibits $\Q_n(\O)$ as
the quotient of $\D(G,S)$ by the action of $\G$.
Through this quotient, $\Q_n(\O)$ acquires an induced ordering and grading with $\rk(\lambda)=n-\ell$.

\begin{theorem}[\textbf{Orbits}]\label{thm:orbits}
Let $S$ be a finite $G$--set with orbit set $\O:=\O(S)$.
Then the $\G$--orbits of $\D(G,S)$ are in bijection with 
the $\O$--labeled partitions of $n$. More explicitly, the fibers
of the surjective map $\pi_n:\D(G,S)\to\Q_n(\O)$, defined below, are $\G$--orbits.
\end{theorem}
\begin{proof}
Write $\O(S) = \{\oo_1,\dots,\oo_k\}$.
For $(\wt{\beta},z)$ with
$\beta=\{B_1,\dots,B_\ell\}$ written so that $|B_1|\geq \cdots\geq |B_\ell|>0$,
define \[ \pi_n(\wt{\beta},z) := 
(|B_1|,\dots,|B_\ell|\;||\;|z^{-1}(\oo_1)|_{\oo_1},\dots,|z^{-1}(\oo_k)|_{\oo_k}).\]
As discussed in the beginning of this subsection,
the group action preserves the list of cardinalities,
thus $\pi_n(\parts)=\pi_n(w.\parts)$ for all $w\in \G$ and $\pi_n$
descends to a well-defined map on orbits. It remains to show that $\pi_n$ is bijective.

Indeed, given a labeled partition 
$\lambda=(\lambda_1,\ldots,\lambda_\ell || \lambda_{\oo_1}\,\ldots,\lambda_{\oo_k})$
there exists a partition of $\n$ with blocks $(B_1,\ldots, B_\ell, Z_1,\ldots,Z_k)$
of respective sizes $|B_i| = \lambda_i$ and $|Z_j|=\lambda_{\oo_j}$.
This will give a partial $G$-partition with zero block $Z=\cup Z_j$ once we assign colorings:
for every $1\leq i \leq \ell$ color $B_i$ with the constant function $b_i\equiv 1\in G$;
and for the zero block, pick orbit representatives $s_j\in \oo_j$ and define $z:Z\to S$
by mapping $Z_j$ to $s_j$ for every $1\leq j\leq k$.
The resulting partial $G$-partition $(\wt{\beta}_\lambda,z_\lambda)$ maps to $\lambda$ under $\pi_n$.

To see that $\pi_n$ is injective, suppose $(\wt{\alpha},z_\alpha)$ also maps to $\lambda$.
We construct an element $(g,\sigma)\in \G$ transporting
$(\wt{\beta}_\lambda,z_\lambda)\mapsto(\wt{\alpha},z_\alpha)$.
There exists some permutation $\sigma\in \sym$ such that $\sigma.B_i = A_i$ for all $i$ and
$\sigma.z_{\beta}^{-1}(\oo_j) = z_{\alpha}^{-1}(\oo_j)$ for all $j$, as these sets
have equal size respectively. Next, fix representative colorings $a_i:A_i\to G$
for $\wt{A}_i$. Define $g = (g_1,\ldots,g_n)\in G^n$ as follows: if $r\in B_i$ take
$g_r= a_i(\sigma(r))$; otherwise $r\in Z_j$ so $z_{\lambda}(r),z_{\alpha}(\sigma(r))$ belong
to the same orbit $\oo_j$ and we take $g_r$ to be such that
$g_r.z_\lambda(r) = z_\alpha(\sigma(r))$.
\end{proof}

\begin{example}[\textbf{Type C Dowling poset}]\label{ex:typeCorbits}
Recall from Example \ref{ex:typeC} the poset $\D(G,S)$ where $G=\Z_2$ acts on
$S$ trivially. The set of orbits, viewed as a quotient of $\D[2](G,S)$ by the
action of $\G[2]$, is depicted in Figure \ref{fig:OZ2} for $S=\{\pm1\}$ 
(the figure is also labeled by objects that will be explained in Example \ref{ex:typeCP}).

\begin{figure}[hbt]
\begin{tikzpicture}[scale=.75]
\foreach \x/\y in
{-0.5/-3,-6.5/1,-4.9/1,-1.3/1,0.3/1,3.9/1,5.5/1,-5.7/5,-1.3/5,0.3/5,4.7/5}
{
\draw[dashed,-]
(\x,\y)--(\x,\y+1)--(\x+1,\y+1)--(\x+1,\y)--(\x,\y)--(\x+1,\y+1)--(\x+1,\y)--(\x,\y+1);
\draw[dashed,-] (\x+0.5,\y)--(\x+0.5,\y+1);
\draw[dashed,-] (\x,\y+0.5)--(\x+1,\y+.5);
}
\foreach \x in {-5.2,0,5.2}
{
\draw[-] (\x,0.8)--(0,-1);
}
\foreach \x in {-5.2,0}
{
\draw[-] (-5.2,4.8)--(\x,3);
\draw[-] (5.2,4.8)--(-\x,3);
}
\foreach \x in {-5.2,5.2}
{
\draw[-] (0,4.8)--(\x,3);
}
\draw[ultra thick,cyan,-] (-6.5,1)--(-6.5,2);
\draw[ultra thick,cyan,-] (-4.9,1)--(-3.9,1);
\draw[ultra thick,green,-] (-1.3,1)--(-0.3,2);
\draw[ultra thick,green,-] (0.3,2)--(1.3,1);
\draw[ultra thick,blue,-] (3.9,1.5)--(4.9,1.5);
\draw[ultra thick,blue,-] (6,1)--(6,2);
\draw[fill,orange] (-5.7,5) circle (2pt);
\draw[fill,red] (-1.3,5.5) circle (2pt);
\draw[fill,red] (0.8,5) circle (2pt);
\draw[fill,magenta] (5.2,5.5) circle (2pt);
\fill[nearly transparent] (-0.5,-3) rectangle (0.5,-2);
\node at (0,-1.5) {$(1,1\,||\,0)$};
\node at (-5.2,2.5) {$(1\,||\,1_1)$};
\node at (0,2.5) {$(2\,||\,0)$};
\node at (5.2,2.5) {$(1\,||\,1_{-1})$};
\node at (-5.2,6.5) {$(0\,||\,2_1)$};
\node at (0,6.5) {$(0\,||\,1_1,1_{-1})$};
\node at (5.2,6.5) {$(0\,||\,2_{-1})$};
\end{tikzpicture}
\caption{The labeled partitions are the $\sym[2][\Z_2]$-orbits of the type C toric Dowling poset $\D[2](\Z_2,\{\pm1\})$ from Figure \ref{fig:DZ2-toric} and Examples \ref{ex:typeC}, \ref{ex:typeCorbits}. This figure will be further explained in Example \ref{ex:typeCP} and Figure \ref{fig:PZ2-toric}.}
\label{fig:OZ2}
\end{figure}

\end{example}

\begin{example}[\textbf{Hexagonal Dowling poset}]\label{ex:Z6orbits}
Recall the poset $\D(G,S)$ from Example \ref{ex:Z6} where $G=\Z_6$ acts on
$S=\{e,z_1,z_2,z_3,w_1,w_2\}$. 
Since $\O(S)=\{\oo(e),\oo(z),\oo(w)\}$, the orbits are in bijection with 
$\O$--labeled partitions of $n$ where $\O$ is a 3-element set.
The set of orbits when $n=2$, viewed as a quotient of
$\D[2](G,S)$ by the action of $\G[2]$, are depicted in Figure \ref{fig:OZ6-OZ4}.

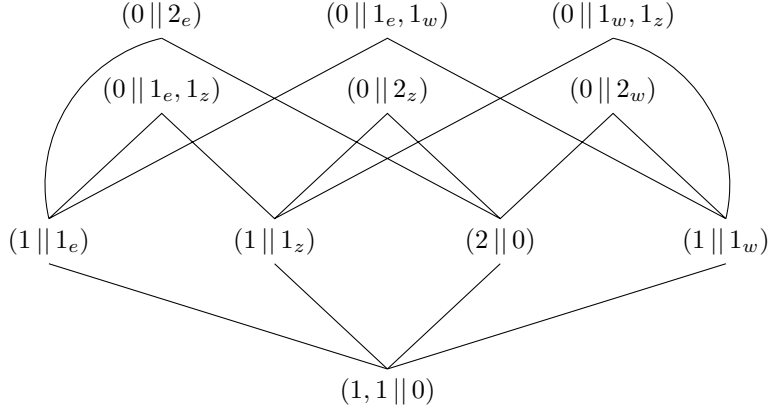
\begin{figure}[htb]
\begin{tikzpicture}
\node at (0,-1.5) {$(1,1\,||\,0)$};
\node at (-4.5,0.5) {$(1\,||\,1_e)$};
\node at (-1.5,0.5) {$(1\,||\,1_z)$};
\node at (1.5,0.5) {$(2\,||\,0)$};
\node at (4.5,0.5) {$(1\,||\,1_w)$};
\node at (0,2.5) {$(0\,||\,2_z)$};
\node at (0,3.5) {$(0\,||\,1_e,1_w)$};
\node at (-3,2.5) {$(0\,||\,1_e,1_z)$};
\node at (-3,3.5) {$(0\,||\,2_e)$};
\node at (3,2.5) {$(0\,||\,2_w)$};
\node at (3,3.5) {$(0\,||\,1_w,1_z)$};
\foreach \x in {-4.5,-1.5,1.5,4.5} {\draw[-] (\x,0.2)--(0,-1.2);}
\draw[-] 
(-4.5,0.8)--(-3,2.2)--(-1.5,0.8)--(0,2.2)--(1.5,0.8)--(3,2.2)--(4.5,0.8)
(-4.5,0.8)edge[bend left=45](-3,3.2)
(4.5,0.8)edge[bend right=45](3,3.2)
(-4.5,0.8)--(0,3.2)--(4.5,0.8)
(-3,3.2)--(1.5,0.8)
(3,3.2)--(-1.5,0.8);
\end{tikzpicture}
\caption{Hasse diagram for the quotient of the hexagonal Dowling poset 
$\D[2](G,S)$ by $\G[2]$, where $G=\Z_6$ and $\O(S)=\{e,z,w\}$.
If we replace all of the $w$'s with $t$'s, this 
is \textit{also} the Hasse diagram for the quotient of the square
Dowling poset $\D[2](G,S)$ by $\G[2]$, where $G=\Z_4$ and $\O(S)=\{e,z,t\}$.
See Examples \ref{ex:Z6}, \ref{ex:Z4}, \ref{ex:Z6orbits}, \ref{ex:Z4orbits},
\ref{ex:Z62}, and \ref{ex:Z42}.}
\label{fig:OZ6-OZ4}
\end{figure}
\end{example}

\begin{example}[\textbf{Square Dowling poset}]\label{ex:Z4orbits}
Recall the poset $\D(G,S)$ from Example \ref{ex:Z4} where $G=\Z_4$ acts on
$S=\{e,z_1,z_2,t\}$. Here, $\O(S)=\{\oo(e),\oo(z),\oo(t)\}$, and hence the orbits are
in bijection with $\O$--labeled partitions of $n$ for a 3-element set $\O$. This means
that the orbits of the square Dowling poset are in bijection with the orbits of the 
hexagonal Dowling poset, and are thus also depicted in Figure \ref{fig:OZ6-OZ4} for $n=2$.
\end{example}

\subsection{Whitney homology}\label{sec:whitney}

The homology of partition lattices has been studied as a representation of the
symmetric group by Stanley \cite{Stanley1982},
Lehrer and Solomon \cite{LS1986}, Barcelo and Bergeron
\cite{BB1990}, and Wachs \cite{Wachs1998}. 
For Dowling lattices, this has been carried out by Hanlon \cite{Hanlon1984} 
and Gottlieb and Wachs \cite{GW2000}.
Since the partition and Dowling lattices are geometric, it follows from Folkman
\cite{Folkman1966} that they are Cohen Macaulay and hence $\H_i(\D(G))=0$
unless $i=n-2$. 
Dowling \cite{Dowling1973} computed the Whitney numbers for the
lattices $\D(G)$ and shows that 
\[\dim\H_{n-2}(\D(G))=\prod_{j=1}^{n-1}(1+j|G|).\]
 This number can be computed from the characteristic
polynomial in Theorem \ref{thm:charpoly} via its characterization as $(-1)^n\chi(\D(G);0)$.

Since the poset $\D(G,S)$ does not have a unique maximum element when $|S|>1$, it
is interesting to investigate the homology of the poset 
obtained by adding a maximum $\onehat$ of rank $n+1$ to obtain a bounded poset
$\widehat{\D[]}_n(G,S)$.
Delucchi, Girard, and Paolini \cite{DGP2017} prove that
$\widehat{\D[]}_n(\Z_2,S)$ is EL-shellable, hence Cohen-Macaulay, 
when $\Z_2$ acts trivially on $S$. Based on this, we suggest the following.
\begin{conjecture}[\textbf{Shellability}]
\label{conj:shellable}
All $\widehat{\D[]}_n(G,S)$ are shellable. In particular, using Theorem
\ref{thm:charpoly}, the order complex of 
${\D[]}_n(G,S)\setminus\{\zero\}$ 
is homotopy equivalent to a wedge of 
\[\prod_{i=0}^{n-1} (|S|+|G|i-1)\] spheres of dimension $n-1$,
when $S$ is nonempty.
\end{conjecture}
\begin{remark}
Since the first draft of this paper, Paolini \cite{paolini} has proven  Conjecture \ref{conj:shellable}.
Here we focused on the case $S\neq\emptyset$ for simplicity; one can similarly count the spheres and their dimension when $S=\emptyset$ from the characteristic polynomial in Theorem \ref{thm:charpoly}.
\end{remark}

In this paper, we are interested in the Whitney homology, defined by
\begin{equation}\label{eq:whitney}
\WH_r(\D(G,S)) := \bigoplus_{\parts\in\D(G,S)} \H_{r-2}(\zero,\parts).
\end{equation}
Since the intervals $[\zero,\parts]$ are products of partition and Dowling
lattices (Theorem \ref{thm:intervals}) hence geometric lattices, we obtain:
\begin{corollary} \label{cor:whitney-homology-ranks}
The group $\H_{r-2}(\zero,\parts)$ is trivial unless $\rk\parts=r$.
Thus, the summation in $\WH_r(\D(G,S))$ includes only the rank--$r$ elements of $\D(G,S)$.
\end{corollary}

The explicit decomposition of intervals $[\zero,\parts]$ from Theorem
\ref{thm:intervals}\eqref{eq:lowerint} also gives a formula
for the dimension of the Whitney homology, but this can alternatively be derived from our characteristic polynomial calculation:
\begin{theorem}\label{thm:WhitneyHilbert}
The Hilbert series for Whitney homology is
\[P_{\WH}(t):=\sum_{r\geq0}\dim\WH_r(\D(G,S))t^r = 
\prod_{i=0}^{n-1}\left(1+(|S|+|G|i)t\right).\]
\end{theorem}
\begin{proof}
This can be obtained through the following specialization of the characteristic
polynomial in Theorem \ref{thm:charpoly}: 
\[P_{\WH}(t)=(-t)^n\chi\left(\D(G,S);-\frac{1}{t}\right) = 
(-t)^n\prod_{i=0}^{n-1}\left(-\frac{1}{t} - |S| - |G|i\right)\]
when $S$ is nonempty.
This is because the Whitney numbers (of the first kind) are simultaneously the coefficients 
in the characteristic polynomial and (up to sign and reordering) the rank of the Whitney 
homology groups.

When $S=\emptyset$, using the second formula of Theorem \ref{thm:charpoly} yields the result.
\end{proof}

We aim to describe the Whitney homology as an $\G$--module. We do so by using 
familiar representations and hence we fix the following notation:
\begin{itemize}
\item $\iota_j$ is the trivial representation of $\sym[j]$.
\item $\epsilon_j$ is the sign representation of $\sym[j]$.
\item $\pi_m$ is the $\sym[m]$--module structure on the nontrivial
homology group of the partition lattice, $\wt{\Ho}_{m-3}(\Q_{\n[m]})$.
By Stanley's work \cite[Thm.~7.3]{Stanley1982}, this representation, up to a twist by a sign representation, is an induced representation $\operatorname{Ind}_{\langle c_m \rangle}^{\sym[m]}\chi_m$ where $c_m$ is an $m$-cycle and $\chi_m$ is a faithful character.
\item If $H$ is a finite group, then $\rho_m(H)$ denotes the $\sym[m][H]$--module structure on the
nontrivial homology group of a Dowling lattice, $\wt{\Ho}_{m-2}(\D[m](H))$, 
whose character was described by Hanlon \cite[Thm.~3.4]{Hanlon1984}.
\item If $H$ is a group, and $U$ and $V$ are $\sym[n]$-- and $H$--modules, respectively, then 
$U[V]:=U\otimes(V^{\otimes n})$ is naturally a representation of the wreath product group
$\sym[n][H]$.
\end{itemize}

Let $S$ be a finite $G$-set with orbits $\O(S)=\{\oo_1,\dots,\oo_k\}$,
and for each $i$ let $G_i$ be the stabilizer subgroup for a fixed representative $s_i\in\oo_i$.
Given an $\O(S)$--labeled partition $(\lambda||m_1,\dots,m_k)$ of $n$, with 
$\lambda=1^{a_1}2^{a_2}\cdots n^{a_n}$, construct a subgroup of $\G$ by
\[\sym[(\lambda||m_1,\dots,m_k)]:=
\sym[a_1][\sym[1]\times G]\times\cdots\times\sym[a_n][\sym[n]\times G]
\times \sym[m_1][G_1]\times\cdots\times\sym[m_k][G_k]\]
and construct a $\sym[(\lambda||m_1,\dots,m_k)]$--representation $V_{(\lambda||m_1,\dots,m_k)}$ by
\[\iota_{a_1}[\pi_1]\otimes \epsilon_{a_2}[\pi_2]\otimes \cdots\otimes
\epsilon_{a_{2i}}[\pi_{2i}]\otimes\iota_{a_{2i+1}}[\pi_{2i+1}]\otimes\cdots
\otimes \rho_{m_1}(G_1)\otimes\cdots\otimes\rho_{m_k}(G_k),\]
where $\pi_k$ is viewed as a representation of $\sym[k]\times G$ with $G$ acting trivially.

Inducing such modules to $\G$ then allows us to describe the $\G$--module structure on the Whitney homology of $\D(G,S)$. This is done in the following theorem, generalizing a result of
 Lehrer and Solomon  \cite[Thm.~4.5]{LS1986} for partition lattices.
\begin{theorem}\label{thm:representation}
As an $\G$--module, \[\WH_r(\D(G,S))\cong\bigoplus 
\Ind_{\sym[(\lambda||m_1,\dots,m_k)]}^{\G} V_{(\lambda||m_1,\dots,m_k)},\]
where the sum is over all $\O(S)$--labeled partitions of $n$ such that $\ell(\lambda)=n-r$.
\end{theorem}
\begin{proof}
By our description of the orbits in Theorem \ref{thm:orbits}, it is clear that
the representation decomposes over the $\O(S)$--labeled partitions, and by Corollary \ref{cor:whitney-homology-ranks} all such partitions must have
$\ell(\lambda)=n-r$.
Let us consider a single orbit corresponding to $(\lambda||m_1,\dots,m_k)$, which
by Theorem \ref{thm:intervals} contributes summands of the form
\[(\wt{\Ho}_{-2}(\Q_{\n[1]}))^{\otimes a_1} \otimes \cdots \otimes
(\wt{\Ho}_{n-3}(\Q_{\n[n]}))^{\otimes a_n} \otimes
\wt{\Ho}_{m_1-2}(\D[m_1](G_1))\otimes \cdots\otimes
\wt{\Ho}_{m_k-2}(\D[m_k](G_k)).\]
The stabilizer of this summand is $\sym[(\lambda||m_1,\dots,m_k)]$, and as a 
representation of the stabilizer it is $V_{(\lambda||m_1,\dots,m_k)}$.
\end{proof}

\section{Symmetric arrangements}\label{sec:arrangements}

Throughout this section, one may take the word `space' to mean either a CW-complex or an algebraic variety over some algebraically closed field. Fix a finite group $G$, acting almost freely on a connected space $X$. By this we mean that the set of singular points of the $G$--actions, i.e. the set of points with non-trivial stabilizer, is finite. Denote this singular set by
\[S:= \operatorname{Sing}_G(X) = \bigcup_{g\in G\setminus\{e\}} X^g.\]
Note that $G$ acts freely on $X\setminus S$, and also the action of
$G$ on $X$ restricts to an action of $G$ on the set $S$. Denote the set of
$G$--orbits of $S$ by
\[\O:= \{G.x\ |\ x\in S\}.\]

\subsection{Introducing the arrangements}\label{sec:arrangementdefs}
We define an arrangement $\A=\A(G,X)$ as the collection of the
following (closed) subspaces in $X^n$:
\begin{enumerate}
\item $H_{ij}(g)=\{(x_1,\dots,x_n)\in X^n\ |\ g.x_i=x_j\}$ for $1\leq
i<j\leq n$ and $g\in G$, and
\item $H_i^s=\{(x_1,\dots,x_n)\in X^n\ |\ x_i=s\}$ for $1\leq i\leq n$ and
$s\in S$.
\end{enumerate}
Note that for the first type of subspace, we may extend our notation to allow
$j<i$ by observing $H_{ij}(g)=H_{ji}(g^{-1})$.

The wreath product group $\G$ acts naturally on $X^n$: first, the group $G^n$ acts on $X^n$ coordinatewise, and then $\sym$ permutes the coordinates.
Explicitly, $w=(g_1,\dots,g_n,\sigma)\in \G$ acts on $(x_1,\dots,x_n)\in X^n$ by sending it to the tuple whose $\sigma(i)$-th entry is $g_i.x_i$. This action induces an action on $\A(G,X)$ 
as described here:
\begin{enumerate}
\item $w.H_{ij}(h) = H_{\sigma(i)\sigma(j)}(g_jhg_i^{-1})$, and
\item $w.H_i^s = H_{\sigma(i)}^{g_i.s}$.
\end{enumerate}

Because the arrangement $\A(G,X)$ is $\G$--invariant, the action of $\G$ on $X^n$
also restricts to an action of $\G$ on the complement, which we denote by 
\[\M(G,X) := X^n\setminus \bigcup_{H\in\A} H.\]
In fact, the action of $\G$ on $\M(G,X)$ is free.

\begin{remark}\label{rmk:orbitconfig}
The space $\M(G,X)$ is what we denote by $\Conf_n^G(X\setminus S)$ in the
introduction. We use this notation to emphasize our perspective of studying
$\M(G,X)$ as a subspace of $X^n$ rather than $(X\setminus S)^n$.
\end{remark}

\subsection{The poset of layers}\label{sec:layers}

A fundamental combinatorial object attached to the arrangement $\A$ is its \emph{poset of layers}.

\begin{definition}
A \textbf{layer} of $\A$ is defined to be a connected component
of an intersection $\cap_{H\in B}H$ for some subset $B\subseteq\A$.

The set of all layers is partially ordered under reverse
inclusion. We call the resulting poset the \textbf{poset of layers} and denote it by $\P(G,X)$.
\end{definition}
The poset $\P(G,X)$ has a unique minimum $\zero$ corresponding to the empty intersection $X^n$. Moreover, the action of $\G$ on $X^n$ and $\A(G,X)$ induces an action 
of $\G$ on $\P(G,X)$.

Note that intersections of layers need not be connected; a simple example of
this phenomenon appears in Example \ref{ex:typeCP} below. However, it will prove useful to
observe that the intersection of \emph{enough} subspaces gives a connected space. This fact will
follow from the proof of Theorem \ref{thm:layers} and is a special
property of the type of arrangements we consider here. 

\begin{example}
It can be helpful to keep in mind the classical example of the trivial group
$G=\{1\}$ acting on any space $X$. Here, the arrangement $\A(\{1\},X)$ is
called the braid arrangement in $X^n$. The name comes from its complement
$\M(\{1\},X)$ which is an ordered configuration space whose fundamental group is
a generalized pure braid group. 

Since $S=\emptyset$ in this example, the arrangement $\A$
consists only of subspaces $H_{ij}$ for $1\leq i<j\leq n$. All intersections
will be connected, and the poset of layers is the partition lattice $\Q_{\n}$.
One can see this by viewing an intersection as putting an equivalence relation
on the coordinates of $X^n$, since $(x_1,\dots,x_n)\in H_{ij}$ whenever
$x_i=x_j$. 
\end{example}

\begin{example}[\textbf{Type C Dowling poset}]\label{ex:typeCP}
Let $X$ be one of $\C$, $\C^\times$, or a complex elliptic curve, and let
$G=\Z_2$ act on $X$ by using the group inversion. The arrangements $\A(G,X)$
arise naturally from the type C root system, viewed as characters on the complex
torus, and were studied in \cite{bibby2} where a specialization of Theorem
\ref{thm:layers} appeared as Theorem 1.

For this action of $\Z_2$ on $X$, the set $S$ is the set of two-torsion points $X[2]$. 
More generally, one could consider this action of $\Z_2$ on any algebraic group $X$ with finitely many two-torsion points.

The poset $\P(\Z_2,\C^\times)$ is isomorphic to the poset $\P(\Z_2,S^1)$, where the latter arises from the action of $\Z_2$ on $S^1$ by using the group inversion, and the Hasse diagram in the case $n=2$ is given in Figure \ref{fig:PZ2-toric}.
To foreshadow Theorem \ref{thm:layers} below, one can
see an obvious isomorphism between this poset and
$\D[2](\Z_2,\{\pm1\})$, whose Hasse diagram is shown in Figure \ref{fig:DZ2-toric}.

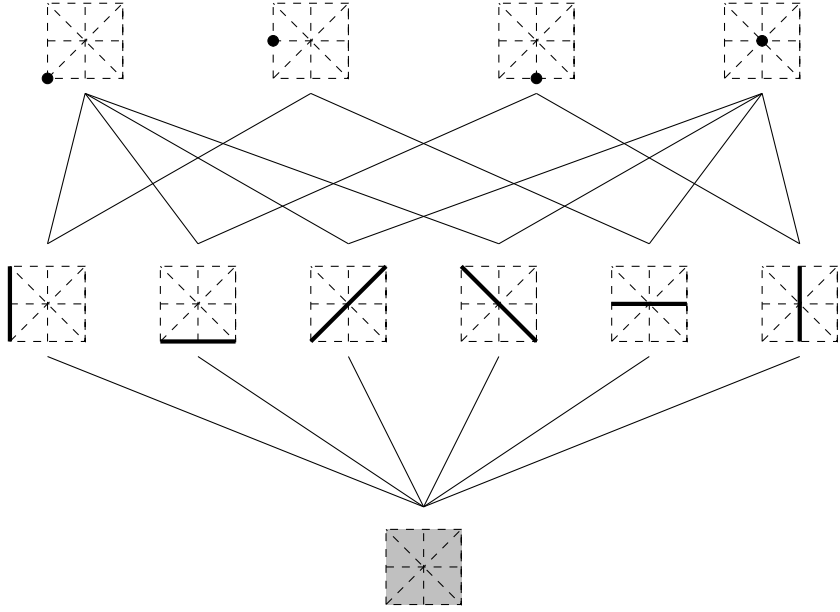
\begin{figure}[htb]
\begin{tikzpicture}[scale=1]
\foreach \x/\y in
{-0.5/-2.5,-5.5/1,-3.5/1,-1.5/1,0.5/1,2.5/1,4.5/1,-5.0/4.5,-2.0/4.5,1.0/4.5,4.0/4.5}
{
\draw[dashed,-]
(\x,\y)--(\x,\y+1)--(\x+1,\y+1)--(\x+1,\y)--(\x,\y)--(\x+1,\y+1)--(\x+1,\y)--(\x,\y+1);
\draw[dashed,-] (\x+0.5,\y)--(\x+0.5,\y+1);
\draw[dashed,-] (\x,\y+0.5)--(\x+1,\y+.5);
}
\foreach \x in {-5,-3,-1,1,3,5}
{
\draw[-] (\x,0.8)--(0,-1.2);
}
\foreach \x in {-5.0,-3.0,-1,1}
{
\draw[-] (-4.5,4.3)--(\x,2.3);
\draw[-] (4.5,4.3)--(-\x,2.3);
}
\draw[-] (1.5,4.3)--(5,2.3);
\draw[-] (-1.5,4.3)--(3,2.3);
\draw[-] (1.5,4.3)--(-3,2.3);
\draw[-] (-1.5,4.3)--(-5,2.3);
\draw[ultra thick,-] (-5.5,1)--(-5.5,2);
\draw[ultra thick,-] (-3.5,1)--(-2.5,1);
\draw[ultra thick,-] (-1.5,1)--(-0.5,2);
\draw[ultra thick,-] (0.5,2)--(1.5,1);
\draw[ultra thick,-] (2.5,1.5)--(3.5,1.5);
\draw[ultra thick,-] (5.0,1)--(5.0,2);
\draw[fill] (-5.0,4.5) circle (2pt);
\draw[fill] (-2.0,5) circle (2pt);
\draw[fill] (1.5,4.5) circle (2pt);
\draw[fill] (4.5,5) circle (2pt);
\fill[nearly transparent] (-0.5,-2.5) rectangle (0.5,-1.5);
\end{tikzpicture}
\caption{Poset of layers for the type C toric arrangement,
$\P[2](\Z_2,\C^\times)$ 
or $\P[2](\Z_2,S^1)$, 
which is isomorphic to $\D[2](\Z_2,\{\pm1\})$ from
Figure \ref{fig:DZ2-toric}. See also Figure \ref{fig:OZ2} and Examples
\ref{ex:typeC} and \ref{ex:typeCP}.}
\label{fig:PZ2-toric}
\end{figure}

To demonstrate the possibility of disconnected intersections, observe 
that when $X=\C^\times$ and $n=2$, the intersection of the two diagonals is a set of
two points: $H_{12}(1)\cap H_{12}(-1)=\{(1,1),(-1,-1)\}$. This can be seen in the Hasse 
diagram since there is not a unique least upper bound of $H_{12}(1)$ and $H_{12}(-1)$.
\end{example}

\begin{remark}
\begin{enumerate}
\item
When $X=\C$, the arrangement $\A(G,X)$ is a (complex) hyperplane arrangement.
Intersections of hyperplanes are all connected, and the poset of layers is a
geometric lattice usually referred to as the intersection lattice.
\item
One may wonder why we consider the set of connected components of
intersections, rather than the set of intersections themselves. 
A case for this choice is made by having codimension be a strictly increasing function on $\P(G,X)$, which is in fact proportional to the intrinsic rank.

Meanwhile, the proof of Theorem \ref{thm:layers} below shows that every layer is in fact an
intersection. That is, the poset of layers is a subposet of the poset of all
intersections, and one that appears easier to work with and study.
It is important to note that this relationship does not exist for general 
arrangements.
\end{enumerate}
\end{remark}

The main theorem of this subsection is that the poset of layers is in fact familiar from our above work.
\begin{theorem}[\textbf{Combinatorial description of the poset of layers}]
\label{thm:layers}
Let $X$ be a 
connected 
almost free $G$-space with 
\[S=\operatorname{Sing}_G(X)=\displaystyle\bigcup_{g\in G\setminus\{e\}} X^g.\]
There is a $\G$--equivariant isomorphism of posets
\begin{equation}\label{eq:posetisom}
\P(G,X) \cong \D(G,S).
\end{equation}
\end{theorem}
Considering just the atoms in these two posets for a moment,
isomorphism \eqref{eq:posetisom} is given by the obvious bijection between
$\alpha_{ij}(g)$, $\alpha_i^s$ (from Lemma \ref{lem:atoms}) and the subspaces
$H_{ij}(g)$, $H_i^s$ in the arrangement $\A(G,X)$.
We make note of some key observations about how these subspaces intersect:
e.g. imposing both $g.x_i=x_j$ and $h.x_i=x_j$ has the consequence that $g^{-1}h.x_i = x_i$,
so $x_i\in X^{g^{-1}h}$ and $x_j\in X^{hg^{-1}}$ are singular points. 
More generally, one can check the following. 
\begin{lemma}\label{lem:intersect}
\begin{enumerate}
\item \label{eq:intersect2}
For $i,j\in\n$ and $g,h\in G$ distinct, the intersection $H_{ij}(g)\cap
H_{ij}(h)$ has connected components $H_i^s\cap H_j^{g.s}$ where $s\in X^{g^{-1}h}$.
\item \label{eq:intersect1}
For $i\in\n$ and $s,t\in S$ distinct, we have $H_i^s\cap H_i^t=\emptyset$.
\item \label{eq:intersect3}
For $i,j,k\in\n$ distinct and $g,h\in G$, 
$H_{ij}(g)\cap H_{j k}(h) \subseteq H_{ik}(hg)$.
\end{enumerate}
\end{lemma}

\begin{proof}[Proof of Theorem \ref{thm:layers}]
Start by defining a map $\phi:\D(G,S)\to\P(G,X)$. For
$\parts\in\D(G,S)$ with partition $\beta=(B_1,\ldots,B_\ell)$ we get a product decomposition
\[ X^{\n} = X^{B_1\cup \ldots\cup B_\ell \cup Z} \cong X^{B_1}\times \ldots\times X^{B_\ell}\times X^Z \]
where $X^{B_i}$ is the space of functions $B_i\to X$. Now, for every projectively colored block $\wt{B}_i\in \wt{\beta}$ we define a connected subspace $X^{\wt{B}_i}\subseteq X^{B_i}$, and one $X^z\subseteq X^Z$ so that their product gives a connected subspace
\[ X^{\parts}:= X^{\wt{B}_1}\times\ldots\times X^{\wt{B}_\ell}\times X^z \;\subseteq\; X^{B_1}\times \ldots\times X^{B_\ell}\times X^Z \cong X^{\n}. \]
Then $\phi$ will be defined to be
\[ \phi: \parts\mapsto X^{\parts}.\]

To illustrate the construction of $X^{\wt{B}_i}$, consider first a projectivized
$G$--coloring on $\{1,\ldots,d\}$ denoted by the suggestive notation
$[g_1:\ldots:g_d]$. This defines a subspace $X^{[g_1:\ldots:g_d]}\subseteq X^d$ by imposing
the equations $(g_i^{-1}.x_i = g_j^{-1}.x_j)$ for all $i$ and $j$. Then the map $\iota_g: X\into X^d$ given by
\[
\iota_g: x\mapsto (g_1.x, g_2.x, g_3.x,\ldots, g_d.x)
\]
gives an isomorphism $X \cong X^{[g_1:\ldots:g_d]}$, and thus the space we defined is connected.
Note that this is well defined, as a different representative $(g_1h,\ldots,g_dh)$
gives equations $h^{-1}g_i^{-1}.x_i = h^{-1}g_j^{-1}.x_j$ which are clearly
equivalent constraints. Also note that the set of defining equations has
many redundancies, and in fact it suffices to only consider pairs $(i,j)$
with $i=1$. However, we avoid making such choices for the purpose of having a canonical construction.

For a general block $\wt{B}\in \wt{\beta}$ we follow the same procedure: pick a representative coloring $b:B\to G$ and consider the subspace $X^{\wt{B}}\subseteq X^B$ consisting of functions $x_B: B\to X$ satisfying
\[ b(i)^{-1}.x_B(i) = b(j)^{-1}.x_B(j) \quad \forall \, i,j\in B.\]
As in the previous paragraph, this definition gives a connected space isomorphic to $X$, and independent of the choice of representative coloring $b$.  

As for the zero block, the function $z:Z\to S\subset X$ gives a point in $X^Z$,
which we take to be the (connected) space $X^z$. One should consider the special case $Z=\{1,\ldots,d\}$, in which this point will be
\[ (z(1),\ldots,z(d)) \in X^d. \]

Observe that the subspace $X^{\parts}$ is cut out of
$X^n$ by equations of the form $g.x_i=x_j$ and $x_i=s$. This gives an
alternative description of $X^{\parts}$ as the intersection of atomic layers
$H_{ij}(g)$ and $H_i^s$. Explicitly, for every $\wt{B}\in\wt{\beta}$ define 
the arrangement
\[ \A[](\wt{B}):=\{H_{ij}(g)\ |\ i,j\in B, gb(i)=b(j)\} \]
and let
\[ \A[](z):=\{H_i^s\ |\ i\in Z, z(i)=s\}.\]
Denote the union of these arrangements by $\A[]\parts$, and observe that $X^{\parts}$ is precisely the
intersection of the subspaces in $\A[]\parts$.

With the definition of $X^{\parts}$ in hand, the map $\phi$ is defined. We must  show that $\phi$ is order-preserving, $\G$-equivariant, and bijective.

Now equivariance is quick: For $w\in\G$ and $\parts\in\D(G,S)$, we have
$\A[](w.\parts)=w.\A[]\parts$ and hence  $X_{w.\parts} = w.X_{\parts}$.

To see that $\phi$ is order-preserving, we consider our two covering
relations from Definition \ref{def:poset}. 
First, consider a \emph{merge}
$(\wt{\beta}\cup\{\wt{A},\wt{B}\},z)\prec(\wt{\beta}\cup\{\wt{C}\},z)$
where $C=A\cup B$ and $c=a\cup b g$ for $g\in G$. Then $X^C\cong X^A\times X^B$ and its subspace $X^{\wt{C}}\subseteq X^C$
is defined by equations of the form $(c(i)^{-1}.x(i)= c(j)^{-1}x(j))$. When $i,j\in A$ these are the defining equations of $X^{\wt{A}}\subseteq X^A$, and when $i,j\in B$ the analogous statement holds for $X^{\wt{B}}$. Therefore, every point $x_C\in X^{\wt{C}}$ satisfies the conditions of being in $X^{\wt{A}}\times X^{\wt{B}}$, and the inclusions $X^{\wt{C}}\subseteq X^{\wt{A}}\times X^{\wt{B}}$ and $ X^{(\wt{\beta}\cup\{\wt{C}\},z)}\subseteq X^{(\wt{\beta}\cup\{\wt{A},\wt{B}\},z)}$ follow.

Next, consider a \emph{coloring} $(\wt{\beta}\cup\{\wt{B}\},z)\prec (\wt{\beta},z')$, where $Z'=Z\cup B$ and $f:G\to S$ is an equivariant function such that $z' = z\cup f\circ b$. Then $X^{Z'} \cong X^B\times X^Z$ and the subspace $X^{z'}$ is the point $z'$. Since $z'$ restricts to the function $z$ on $Z$, it maps into the subspace $X^{z}$ under the projection $X^{Z'}\to X^Z$. As for its projection to $X^B$, on every $i,j\in B$ we need to check that $b(i)^{-1}.z'(i) = b(j)^{-1}.z'(j)$. Indeed, by the equivariance of $f$,
\[
b(i)^{-1}.z'(i) = b(i)^{-1}.f(b(i)) = f(1) = b(j)^{-1}.z'(j)
\]
and the inclusion $X^{z'}\subseteq X^{\wt{B}}\times X^{z}$ follows.

Finally, to show that $\phi$ is bijective we will construct an inverse map.
Given $Y\in \P(G,X)$, let $\A[](Y)$ be the subset of $\A(G,X)$ consisting of
subspaces which contain $Y$. Define $Z_Y = \{i\in\n\ |\ H_i^s\in \A[](Y), \text{ some }s\in S\}$. Then by Lemma \ref{lem:intersect}\eqref{eq:intersect1}, we have for
each $i\in Z_Y$ a unique $s\in S$ for which $H_i^s\in\A[](Y)$; this
defines a map $z_Y:Z_Y\to S$.
Next, by Lemma \ref{lem:intersect}\eqref{eq:intersect2}, for $i,j\notin Z_Y$, if there
is some $g\in G$ for which $H_{ij}(g)\in \A[](Y)$ then this $g$ is unique and we
denote it by $g_{ij}$. Define a partition $\beta_Y$ of $\n\setminus Z_Y$ from
the equivalence relation with $i\sim j$ if there is such a $g_{ij}$. Moreover,
for $B\in\beta_Y$, one can always construct a coloring $b:B\to G$ so that $g_{ij}b(i)=b(j)$,
giving a partial $G$-partition $\wt{\beta}_Y$ of $\n$. It is now easy
to check that the assignment $Y\mapsto (\wt{\beta}_Y,z_Y)$ is indeed an
inverse to $\phi$.
\end{proof}

\begin{remark}\label{rmk:free}
As we saw in Remark \ref{rmk:orbitconfig}, the two complements $\M(G,X)$ and $\M(G,X\setminus S)$ coincide.
However, the former is viewed as the complement of an arrangement in $X^n$ while the latter sits 
inside $(X\setminus S)^n$. The two arrangements have rather different combinatorics,
even though their complements are equal: the poset of layers of $\A(G,X\setminus S)$ is
$\P(G,X\setminus S)\cong \D(G,\emptyset)$ since $G$ acts freely on 
$X\setminus S$, while the arrangement in $X^n$ has
$\P(G,X)\cong \D(G,S)$.

The benefit of working with $X$ over $X\setminus S$ is most apparent when $X$ is compact, e.g. a smooth projective variety, while $X\setminus S$ is not. In particular, when the cohomology of $X$ has pure Hodge structure, many spectral sequence calculations simplify greatly.
\end{remark}

\subsection{More examples}\label{sec:arrangementexs}
Our first motivating example is the Dowling lattice. 
Dowling \cite{Dowling1973} showed that when $G$ acts on $\AA^1_k$ linearly,
i.e via a character $G\to k^*$, the lattice $\D(G)$ is the
intersection lattice of $\A(G,\AA^1_k)$. 
In the case that $X=\C$, one considers $G=\mu_d$ the group of
$d$'th roots of unity, and the hyperplanes in
the arrangement are the reflecting hyperplanes for reflections in $\G$. Hence
these are complex reflection arrangements, and their complements are $K(\pi,1)$ spaces (see \cite{Nakamura1983})  
whose cohomology is particularly interesting.

In this subsection we consider examples of interest with varying $X$ and $G$, and
relate them to the examples of posets $\D(G,S)$ explored in Section
\ref{sec:posetexs}.

Recall one of our motivating examples: the type C arrangements, described in Example \ref{ex:typeCP} above. Generalizing on this, one can consider any algebraic group $X$ and
let $G=\Z_2$ act by the group inversion. In this case the set $S$, where inversion fails to be 
free, is the
set $X[2]$ of two-torsion points. Even more generally, one can take any finite subgroup
$G\subseteq\Aut(X)$ of algebraic group automorphisms, for which the set $S$ is
finite, and consider the resulting arrangements $\A(X,G)$.

For concreteness, take $X$ to be a complex elliptic curve.
Most elliptic curves have $\Aut(X)=\Z_2$, and so the type C
elliptic arrangements are the only ones arising as $\A(X,G)$. But when the
$j$-invariant is either 0 or 1728, extra automorphisms appear. We describe the
arrangements arising from the action of the automorphism group in the following
two examples. 

\begin{example}[\textbf{Hexagonal elliptic curve}]\label{ex:Z62}
Let $X$ be a complex elliptic curve with $j(X)=0$.
Alternatively, this is the complex torus $\C/(\Z\oplus\zeta_3\Z)$,
corresponding to a tiling of the plane by equilateral triangles.
Then the group of automorphisms is $G=\Z_6$, which we may consider as generated
by multiplication with a primitive $6$'th root of unity, $-\zeta_3$. 
The set of points where the action fails to be free is
\[S=\{e,z_1,z_2,z_3,w_1,w_2\},\]
represented in $\C$ by $e=0$, $z_1=\frac{1}{2}$,
$z_2=\frac{1}{2}\zeta_3$, $z_3=\frac{1}{2}(1+\zeta_3)$, $w_1=\frac{1}{3}(1+2\zeta_3)$,
and $w_2=\frac{1}{3}(2+\zeta_3)$.
The action of $G$ on $S$ agrees with that in Example \ref{ex:Z6}, hence the
poset of layers is the hexagonal Dowling poset given in that example above.
\end{example}

\begin{example}[\textbf{Square elliptic curve}]\label{ex:Z42}
Let $X$ be a complex elliptic curve with $j(X)=1728$, or alternatively, the complex torus $\C/(\Z\oplus i\Z)$. Then the
automorphism group $G$ is $\Z_4$, which we may consider to be generated by multiplication with the
primitive fourth root of unity $i$.
The points where the action fails to be free are the two-torsion points 
\[X[2]=\{e,z_1,z_2,t\},\]
represented in $\C$ by $e=0$,
$z_1=\frac{1}{2}$, $z_2=\frac{1}{2}i$, and $t=\frac{1}{2}(1+i)$. The group $G$
acts on these points just as it did in Example \ref{ex:Z4}, hence the resulting poset of
layers is the square Dowling poset discussed in that example above.
\end{example}

\begin{example}[\textbf{Translation by torsion points}]\label{ex:translation}
Another interesting example for $X$ an algebraic group is when $d$--torsion
points $G=X[d]$  act by translation. A specific example of this is when
$X=\C^\times$ so that $G=\Z_d$ are the $d$'th roots of unity; here we note that
$\M(\Z_d,\C^\times)=\M(\Z_d,\C)$ is Dowling's motivating example, by Remark \ref{rmk:free}.
The action of $X[d]$ on $X$ is free, and thus the poset of layers for $\A(G,X)$
is always the lattice of $G$-partitions $\D(G,\emptyset)$.
\end{example}

\subsection{Invariant arrangements}\label{sec:subarrangements}

In the above treatment, we construct an arrangement whose complement is the orbit configuration space in $X^{reg}:= X\setminus S$. Next, we consider a variant on this idea, in which one chooses to remove a different collection of points from $X$.

Let $T$ be finite a $G$-invariant subset of $X$, i.e. a finite union of $G$ orbits. Analogously to the discussion above, define an arrangement 
$\A^T(G,X)$  
in $X^n$ consisting of the subspaces:
\begin{enumerate}
\item $H_{ij}(g)$ for $1\leq i<j\leq n$ and $g\in G$, and 
\item $H_i^t$ for $1\leq i\leq n$ and $t\in T$.
\end{enumerate}
Its complement $\M^T(G,X)$  
is the orbit configuration space in $X\setminus T$. We will also be interested in the resulting poset of layers, denoted by $\P^T(G,X)$. 

Because $T$ is $G$-invariant, the arrangement $\A^T(G,X)$  
is $\G$-invariant, and hence $\G$ acts
on $\M^T(G,X)$  
as well as on $\P^T(G,X)$. 
However, note that $\G$ no longer acts freely on 
$\M^T(G,X)$ 
when $T$ does not contain $S$. 

Let us start by mentioning two motivational examples of these invariant
arrangements.

\begin{example}[\textbf{Punctured surface}]\label{ex:punctures}
Suppose that $G$ is a group acting on a Riemann surface $X$. Then if $T$ is a finite 
$G$-invariant subset, the complement of the invariant arrangement $\A^T(G,X)$ 
is the orbit configuration space of the punctured surface. As mentioned in Remark \ref{rmk:free}, 
this is a scenario in which one can benefit from studying the (orbit) configuration space 
of a punctured surface inside of the compact $X^n$ rather than the usual (non-compact) 
$(X\setminus T)^n$.
\end{example}

\begin{example}[\textbf{Types B and D}]\label{ex:BD}
Consider the case that $X=\C$, $\C^\times$, or a complex elliptic curve, and
$G=\Z_2$ acts by using the group inversion, as in Examples
\ref{ex:typeC} and \ref{ex:typeCP}.
The arrangement $\A(X,G)$, discussed in Example \ref{ex:typeCP}, naturally arises from the type C root
system, viewed as characters on a torus. The type B and D root systems also
define arrangements, which are subarrangements of the type C arrangement.
In fact, they are invariant subarrangements, where the type B arrangement 
uses $T=\{e\}\subseteq X[2]$ and in type $D$ we have $T=\emptyset$.
The poset of layers and representation stability for these subarrangements were
studied in \cite{bibby2}.
\end{example}

\begin{remark} It is important to note that while $\D(G,S)$ describes the poset of layers of
$\A(G,X)$, in general it is \emph{not} true that the poset $\D(G,T)$ describes the layers of
$\A^T(G,X)$. 

When $T\subset S$, one immediately sees that the poset $\P^T(G,X)$ 
is a subposet of $\P(G,X)$. However, it is larger than one might at first expect.
In fact, even when $T\cap S = \emptyset$, the singular set $S$ appears in
the description of layers in $\A^T(G,S)$. 
This phenomenon is explained fully next,
in Theorem \ref{thm:sublayers}.
\end{remark}

\begin{theorem}[\textbf{Removing general $T$}]\label{thm:sublayers}
Let $T$ be a finite $G$-invariant subset of 
a connected almost free $G$-space
$X$, and let 
$S = \operatorname{Sing}_G(X)$ as above.
Denote the respective sets of orbits by $\O(T)$ and $\O(S)$.
Then there is a natural equivariant embedding 
\[\P^T(G,X) 
\into \D(G,T\cup S)\]
whose image consists of all pairs $\parts$ for which $|z^{-1}(\oo)|\neq 1$ whenever
$\oo\in\O(S)\setminus\O(T)$. 
\end{theorem}
\begin{proof}
Given $\parts\in\D(G,T\cup S)$, one can construct the layer $X^{\parts}$ of the arrangement $\A^{T\cup S}(G,X)$ 
as in the proof of Theorem \ref{thm:layers}: by intersecting the
subspaces in the collections
\[\A[](\wt{B}):=\{H_{ij}(g)\ |\ i,j\in B, gb(i)=b(j)\}\]
\[\A[](\oo):=\{H_i^{z(i)}\ |\ i\in Z,z(i)\in\oo\}.\] 
We therefore only need to show that if $|z^{-1}(\oo)|\neq1$ for all
$\oo\in\O(S)\setminus\O(T)$, then $X^{\parts}$ 
is in fact a layer of the subarrangement  $\A^T(G,X)$. 

Consider the intersection of subspaces in $\A[](B)$ along with only $\A[](\oo)$ for
orbits $\oo\in\O(T)$. Its connected components are layers of $\A^T(G,X)$,  
by definition. Furthermore,
one of those connected components is indeed  $X^{\parts}$. This last claim
follows from Lemma \ref{lem:intersect}\eqref{eq:intersect2}.

Conversely, consider a layer $Y$ of the arrangement $\A^T(G,X)$  
and the
corresonding pair $(\wt{\beta}_Y,z_Y)\in\D(G,T \cup S)$ from the proof of Theorem
\ref{thm:layers}. We need to show that $|z_Y^{-1}(\oo)|=1$ implies
$\oo\in\O(T)$. Suppose $i\in\n$ is the single element for which
$z_Y(i)\in\oo$. Then there cannot exist another $j\in\n$ and $g\in G$ for which 
$H_{ij}(g)\supseteq Y$, since otherwise Lemma \ref{lem:intersect}\eqref{eq:intersect2} would imply that $z_Y(j)\in \oo$ as well. But then the only way $Y$ can be a layer of $\A^T(G,X)$ 
is if $H_i^{z_Y(i)}\in \A^T(G,X)$, 
which implies $\oo\in\O(T)$.
\end{proof}

Next we discuss the intervals of the arising posets of layers. Theorem \ref{thm:intervals} above shows that closed intervals in $\D(G,S)$ are
products of Dowling and partition lattices. An analogue of this statement is
true for the subposets in question, for which closed intervals are again geometric lattices.
We prove this characterization of intervals in Theorem \ref{thm:sublayers} below, but before we can do that we must
first revisit the case of the Dowling lattice.

\begin{definition}[\textbf{The posets $\DD(G)$}]
Let $\DD(G)$ be the subposet of $\D(G)$ consisting of partial partitions whose
zero block is \emph{not} a singleton.
\end{definition}

Note that $\DD(G)$ is precisely the subposet of $\D(G,\{0\})$ which corresponds 
to  
$T=\emptyset$ in Theorem \ref{thm:sublayers}.
When $\DD(G)$ is a geometric lattice, intervals
built out of these lattices are just as well-behaved as those of our $S$-Dowling posets.
Furthermore, one can extend the description of atoms to this context (Lemma
\ref{lem:atoms}), functoriality (Proposition \ref{prop:functoriality}), and
$\G$--orbits (Theorem \ref{thm:orbits}), but we omit such details. In particular, the $\G$-orbits of $\P^T(G,X)$  
are indexed by $\O(T\cup S)$-labeled partitions $\lambda$ of $n$ for
which $\lambda_\oo\neq1$ whenever $\oo\in\O(S)\setminus\O(T)$.

It remains to determine when the poset $\DD(G)$ is a geometric lattice, which we do in Proposition \ref{prop:geometriclattice} next. 
We note that $\DD(G)$ is not necessarily a geometric lattice when $G$ is trivial, as it fails to be atomic in general. However, we will see in Theorem \ref{thm:subintervals} that this potential difficulty never arises in applications: the stabilizer of any point $s\in\operatorname{Sing}_G(X)$ is nontrivial by definition.

\begin{proposition} \label{prop:geometriclattice}
If $G$ is a nontrivial group, then the poset $\DD(G)$ is a geometric lattice.
\end{proposition}
\begin{proof}
First, observe that if $\alpha$ covers $\beta$ in $\DD(G)$, then $\alpha$ already covers $\beta$ in $\D(G)$. Otherwise there would be some $\gamma\in\D(G)$ with a singleton zero block such that $\beta<\gamma<\alpha$ in $\D(G)$. Since $\gamma<\alpha$ and $\alpha\in\DD(G)$, $\alpha$ has a zero block $Z$ with $|Z|>1$; and since $\gamma>\beta$ with $\beta\in\DD(G)$, $\beta$ has an empty zero block. Thus the relation $\beta < \alpha$ implies that there is a collection of blocks $\{\wt{B}_1,\ldots,\wt{B}_k\}$ in $\beta$ with $k>1$ whose union is $Z$. Merging these blocks into a single one $\wt{B} = \cup \wt{B}_i$, we find an intermediate element $\beta<\alpha\cup\{\wt{B}\}<\alpha$ in $\DD(G)$, contradicting the assumption that $\alpha$ covers $\beta$. This implies that the rank function on $\D(G)$ restricts to a rank function on $\DD(G)$.

To see that $\DD(G)$ is a lattice, consider $\alpha,\beta\in\DD(G)$. Since $\DD(G)$ is a subposet of $\D(G)$, this pair has a least upper bound $\alpha\vee \beta$ and greatest lower bound $\alpha\wedge \beta$ in $\D(G)$.
Since the zero block of $\alpha\vee\beta$ is the union of the zero blocks of $\alpha$ and $\beta$, it follows that $\alpha\vee\beta$ is already the least upper bound in $\DD(G)$.
If $\alpha\wedge\beta$ does not have a singleton zero block, it is the greatest lower bound in $\DD(G)$ as well. Otherwise, if $\alpha\wedge\beta$ has a singleton zero block $\{i\}$, then $(\alpha\wedge\beta)\cup\{i_e\}$ is the greatest lower bound of $\alpha$ and $\beta$ in $\DD(G)$. Note that this element has rank one less than that of $\alpha\wedge\beta$, and so the rank function on $\DD(G)$ is semimodular.

Finally, to see that $\DD(G)$ is atomic (that is, every element is the least upper bound for some set of atoms), we will use our notation of atoms from Lemma \ref{lem:atoms}. 
By our product decomposition of intervals in Theorem \ref{thm:intervals}, we need only consider the case that $\beta\in \DD(G)$ is the empty partition and has only a zero block $Z=\{z_1,\dots,z_\ell\}$. In this case, $\beta$ is the least upper bound of the atoms $\alpha_{ij}(g)$ ranging over $g\in G$ and $1\leq i<j\leq \ell$, since $G$ is not the trivial group.
\end{proof}

\begin{theorem}[\textbf{Local structure for general $T$}]\label{thm:subintervals}
Let $T$ be a finite $G$-invariant subset of $X$, and let $\D^T(G,S)$  
be the 
subposet of $\D(G,T\cup S)$ consisting of pairs $\parts$ for which $|z^{-1}(\oo)|\neq
1$ whenever $\oo\in\O(S)\setminus\O(T)$. 

Then for every $\parts\in\D^T(G,S)$,  
\[\D^T(G,S)_{\leq\parts}  
\cong
\prod_{B\in\beta} \Q_B \times 
\prod_{\oo\in\O(T)} \D[z^{-1}(\oo)](G_\oo) \times
\prod_{\oo\in\O(S)\setminus\O(T)} \DD[z^{-1}(\oo)](G_\oo).\]
\end{theorem}
\begin{proof}
Recall from Theorem \ref{thm:intervals}\eqref{eq:lowerint} that the
decomposition of the interval under $\parts$ in $\D(G,S)$ assigns to some
$(\wt{\alpha},z_\alpha)$ a pair of tuples
$((\alpha_B)_{B\in\beta},(\alpha_\oo)_{\oo\in\O(S)})$. If we require
$|z_\alpha^{-1}(\oo)|\neq1$ whenever $\oo\in\O(S)\setminus\O(T)$, then for 
$\oo\in\O(S)\setminus\O(T)$ the zero block of $\alpha_\oo$ could not be a
singleton.
Thus, the image of $\D(G,S;T)_{\leq \parts}$ under the isomorphism would be the
product decomposition stated in the theorem. 
\end{proof}

Using the notation $\D^T(G,S)$, the product decomposition in Theorem \ref{thm:subintervals} above can be written more succinctly as
\[\D^T(G,S)_{\leq\parts} \cong
\prod_{B\in\beta} \Q_{B} \times 
\prod_{\oo\in\O(S)} \D[{z^{-1}(\oo)}]^{T\cap\{s_\oo\}}(G_\oo,\{s_\oo\}),\]
where for each orbit $\oo$ we fix a representative $s_\oo$. 
In the sequel paper \cite{BG2}, the authors work in this greater generality and simplify notation by using $S$ to mean $S\cup T$ rather than just the set of singular points.

\subsection{Local arrangements}\label{sec:localarr}

This subsection will focus on the local structure of the arrangement $\A(G,X)$.
For this purpose, take the space $X$ to be a \emph{smooth} manifold or variety.
One way to understand what we mean by the local structure of $\A(G,X)$
is to consider a `scanning' procedure:
studying the germ of $\A(G,X)$ at every point $p\in X^n$. The germs 
are what appear in the Leray spectral sequence for the inclusion $\M(G,X)\into X^n$; 
see Theorem \ref{thm:ss} below.

Since each element $\parts\in \D(G,S)$
corresponds to a subspace $X^{\parts}\subseteq X^n$, the incidence relation attaches to 
every point $p\in X^n$ a subposet 
\[
\D(G,S)_p := \{ \parts\in \D(G,S) \mid p\in X^{\parts}\}.
\]
This is exactly the subposet of those layers that meet every neighborhood of $p$.

The collection of layers in $\D(G,S)_p$ is clearly closed under intersection,
and thus has a maximum $(\wt{\beta}_p,z_p)$.
It is also downward-closed (an order ideal), and is therefore the interval $[\zero,(\wt{\beta}_p,z_p)]$ described by Theorem \ref{thm:intervals}.
Geometrically, this characterizes the germ of $\A(G,X)$ at $p$ in $X^n$: it is
well known that when $X$ is a smooth manifold, the restriction of $\A(G,X)$ to a small ball centered at $p$ is isomorphic to a \emph{linear
subspace arrangement} $\A[p]$, whose intersection poset is the interval $[\zero,
(\wt{\beta}_p,z_p)]$. One can see this, e.g. by choosing a Riemannian metric on $X$ and using the exponential map to identify a neighborhood of $p$ with the tangent space $T_pX$ and the linear arrangement therein.

Theorem \ref{thm:intervals} thus translates to the following, 
\begin{theorem}[\textbf{Local arrangements}]\label{thm:localarr} 
For every $p\in
X^n$ the complement of the local arrangement $\A[p]$ is isomorphic to a product
of (free) orbit configuration spaces of points in $\R^d$.

Equivalently, the
restriction of $\Conf_n^G(X\setminus S)$ to any sufficiently small ball is
isomorphic to a product of such orbit configuration spaces.
\end{theorem}
\begin{remark}
This observation is striking for the following two reasons:
\begin{enumerate}
\item The local picture involves orbit configuration spaces for groups 
\emph{different} from $G$, and possibly different actions on $\R^d$ at 
every point.
\item  One could not have avoided the difficulty of this description by 
removing a set of `bad' points, as is typically done with non-free actions.
This is since neighborhoods of these points record their bad behavior. The 
entire description of our posets $\D(G,S)$ was necessary, and before this 
work, the local structure described in Theorem \ref{thm:localarr} was 
generally unknown. This is while a description like above is rather 
important and comes up in applications, e.g. using the Leray spectral 
sequence for $\M(G,X)\into X^n$ for complete $X$.
\end{enumerate}
\end{remark}

Let $\M[p]$ denote the complement of $\A[p]$ inside a small open ball. 
Since $\A[p]$ is a linear subspace arrangement, the work of
Goresky--MacPherson \cite{GM} ties together the cohomology of $\M[p]$ with the
Whitney homology of the interval:
\begin{equation}\label{eq:GM}
\Ho^{(d-1)*}(\M[p]) \cong \WH_{*}(\zero,(\wt{\beta}_p,z_p)).
\end{equation}

\begin{remark}[\textbf{Realizability}]\label{rmk:realizable}
A central question in matroid theory is whether a geometric lattice is realizable by
an arrangement of hyperplanes over a field. Dowling \cite{Dowling1973}
completely settled this for his lattices: he showed that $\D(G)$ is
realizable over $\C$ if and only if $G$ is cyclic, and changing the field
amounts to putting restrictions on which cyclic groups are allowed.

Making contact with our local arrangements, one observes that if
$\D(G,S)$ is the poset of layers for an arrangement of
hypersurfaces in a complex manifold, then each interval, as described in
Theorems \ref{thm:intervals} and \ref{thm:subintervals}, is realizable over
$\C$. Similarly, if $\A$ is an arrangement of smooth hypersurfaces in some smooth algebraic variety over a field $k$, then the tangent spaces to the arrangement at various points are realizations of intervals over $k$.

In particular, Dowling's realizability result implies the following:
\begin{corollary}[\textbf{Restriction on possible stabilizers}]
If an arrangement of
hypersurfaces in a variety $M$ has poset of layer $\cong \D(G,S)$, then the stabilizer subgroup $G_s$ for every $s\in S$ must be
cyclic.
\end{corollary}

By considering $\A(G,X)$ we get, in particular, that if $X$ is any Riemann surface or an algebraic variety over a field $k=\bar{K}$ with an almost free $G$ action, then the stabilizer of any point $x\in X$ must be cyclic.

We wonder about the converse of this:
\begin{question} Suppose that $S$ is a $G$--set for which the
stabilizer subgroups are all cyclic. When can one find an arrangement of
hypersurfaces in a complex manifold whose poset of layers is $\cong \D(G,S)$?

Furthermore, if such an arrangement exists, must it be of the form $\A(G,X)$?
\end{question}
\end{remark}

\subsection{Cohomology of the complement}\label{sec:specseq}

This subsection is devoted to computing $\Ho^*(\M(G,X))$ as explicitly as possible. In 
the sequel to this paper \cite{BG2}, 
we will apply ideas from representation stability to analyze the sequence of $\G$--representations $\Ho^i(\M(G,X))$ as $n$ varies. Recall from the introduction our running assumption on coefficients: When $X$ is a CW complex $\Ho^*(\bullet)$ denotes singular cohomology with coefficients in any ring, and when $X$ is an algebraic variety it denotes $\ell$-adic cohomology with coefficients in either $\Z_\ell$ or $\QQ_\ell$.

In the case of where $G$ acts linearly on $X=\AA^d$, the cohomology of the complement $\M(G,X)$ can be
described in a purely combinatorial way, as the Whitney homology of the
intersection poset as we saw in \eqref{eq:GM}, by the work of Goresky--MacPherson \cite{GM}.
In general, though, there is a spectral sequence converging to 
$\Ho^*(\M(G,X))$, which combines the cohomology of the ambient space 
$X^n$ with the Whitney homology of the poset of layers. 
This spectral sequence 
can be realized as the Leray spectral sequence for the inclusion 
$\M(G,X)\into X^n$ in some cases. We refer the reader to a 
clever construction by Petersen, which has appeared in various 
special cases beforehand but applies in the general context -- 
see \cite{Petersen2017}.

The poset of layers $\P(G,X)$ gives rise to a stratification of the space $X^n$, and
since intervals in the poset are Cohen--Macaulay, we can use Example 3.10 in \cite{Petersen2017} to simplify the $E_1$
page of the spectral sequence. Moreover, in the case that $X$ is projective, one can use the pure Hodge
structure to conclude that most differentials must vanish, similar to Totaro's argument in \cite{Totaro1996}.
It is also worth noting that one could apply Verdier or Poincar\'e duality to
the following theorem to obtain a sequence more akin to that of
Totaro, avoiding compactly supported cohomology. 
\begin{theorem}[\textbf{Spectral sequence for $\Ho^*(\M(G,X))$}]\label{thm:ss} 
There is a spectral sequence
\[E_1^{pq} = 
\bigoplus_{\substack{\parts\in\D(G,S)\\ \rk\parts=p}}  
\wt{\Ho}{}^{p-2}(\zero,\parts)\otimes
\Ho^q_c(X^{\parts}) \Longrightarrow \Ho_c^{p+q}(\M(G,X)).\]
When $X$ is a smooth projective variety, this sequence degenerates at 
the $E_2$ page.
\end{theorem}
The main contributions of this theorem to Petersen's general construction are the following:
\begin{itemize}
\item Every layer $\parts$ contributes exactly one term to $E_1^{pq}$ with homological degrees matching $(p,q)$, as opposed to summing over all combinations $\Ho^i\otimes \Ho^j_c$ with $i+j+2 = p+q$.
\item One gets control over weights. For example, when $X$ is a smooth complex projective variety then 
$E_1^{}$ will have pure weight.
\item Our Theorem \ref{thm:intervals} gives an exact description of $[\zero,\parts]$ and its cohomology.
\end{itemize}

Next, we use our combinatorial formula for the characteristic polynomial in Theorem
\ref{thm:charpoly} to compute the Hilbert series for the above $E_1$ page, 
as we did for the Whitney homology in Theorem \ref{thm:WhitneyHilbert}.
Recall that a specialization of the Hilbert series
($t=u=-1$) computes the Euler characteristic for the complement $\M(G,X)$.
Note that our formula only works for the arrangements $\A(G,X)$, and not the
invariant subarrangements discussed in Subsection \ref{sec:subarrangements}.
Unfortunately, the characteristic polynomial for the subposets 
$\D^T(G,S)$ 
does not in general 
have integer roots. 

\begin{proposition}[\textbf{Hilbert series and Euler number}]\label{prop:poincare-poly}
The Hilbert series of the $E_1$ term of Theorem
\ref{thm:ss} is
\[ \sum_{p,q} (\dim E_1^{pq}) t^pu^q 
= \prod_{i=0}^{n-1} \left(P_c(X;u) + (|S|+|G|i)t\right), \]
where $P_c(X;u)$ is the compactly supported Poincar\'e polynomial of $X$.
\end{proposition}
\begin{proof}
First, assume that $S$ is nonempty. Using in turn the decomposition of Theorem \ref{thm:ss} and the K\"unneth formula relating the compactly supported Poincar\'e polynomials of $X$ and $X^{\parts}\cong X^{n-p}$, we obtain:
\begin{align*}
\sum_{p,q} (\dim E_1^{pq})t^pu^q
&= \sum_{p,q} \sum_{\substack{\parts\in\D(G,S)\\ \rk\parts=p}}  \dim(\wt{\Ho}{}^{p-2}(\zero,\parts))\dim(
\Ho^q_c(X^{\parts})) t^pu^q\\
&= \sum_{p} \sum_{\substack{\parts\in\D(G,S)\\ \rk\parts=p}}  \dim(\wt{\Ho}{}^{p-2}(\zero,\parts))t^p(P_c(X;u))^{n-p}
\end{align*}
Now, since $\dim(\wt{\Ho}{}^{p-2}(\zero,\parts))=(-1)^{\rk\parts}\mu\parts$, we may relate this to the characteristic polynomial, as follows.
\begin{align*}
&= \sum_{p} \sum_{\substack{\parts\in\D(G,S)\\ \rk\parts=p}}  (-1)^p\mu\parts t^p(P_c(X;u))^{n-p}\\
&= (-t)^n \sum_{\parts\in\D(G,S)}  \mu\parts\left(-\frac{P_c(X;u)}{t}\right)^{n-\rk\parts}\\
&= (-t)^n \chi\left(\D(G,S);-\frac{P_c(X;u)}{t}\right)
\end{align*}
Finally, applying Theorem \ref{thm:charpoly}, this is equal to
\[(-t)^n \prod_{i=0}^{n-1} \left(-\frac{P_c(X;u)}{t}-|S|-|G|i\right)
= \prod_{i=0}^{n-1} (P_c(X;u) + (|S|+|G|i)t).\]
When $S=\emptyset$, the poset $\D(G,\emptyset)$ has rank $n-1$, and we modify the last few steps to obtain
\begin{align*}
(-t)^{n-1}P_c(X;u) &\sum_{\parts\in\D(G,\emptyset)}  \mu\parts\left(-\frac{P_c(X;u)}{t}\right)^{n-1-\rk\parts}\\
&= (-t)^{n-1}P_c(X;u) \chi\left(\D(G,\emptyset);-\frac{P_c(X;u)}{t}\right)\\
&= (-t)^{n-1}P_c(X;u) \prod_{i=1}^{n-1} \left(-\frac{P_c(X;u)}{t}-|G|i\right)\\
&= \prod_{i=0}^{n-1} (P_c(X;u) + |G|it).
\end{align*}
\end{proof}

\begin{example}
The complement of the type C toric arrangement $\A(\Z/2,\C^\times)$ has
Poincar\'e polynomial
\[\prod_{i=0}^{n-1} \left(1+t+(2+2i)t\right)
=\prod_{i=1}^{n} \left(1+(1+2i)t\right).\]
This can be seen from using an analogue of Proposition
\ref{prop:poincare-poly} for the Leray spectral sequence, so that $P(u)=1+u$, the
(ordinary) Poincar\'e polynomial. Since this sequence has no nontrivial
differentials, the $E_1$ term gives exactly the cohomology groups.
This method can be used anytime that the Leray spectral sequence degenerates
immediately; for toric arrangements this formula could also be derived by the
work of Moci \cite{Moci2012,Moci2008}.
\end{example}

From Proposition \ref{prop:poincare-poly}, one can now compute the Euler characteristic of the space $\M(G,X)$ by substituting $u=t=-1$. However, a more conceptual and flexible approach to this calculation is presented next.

\subsection{Motive of the complement} \label{sec:motive}
Recall that the compactly supported Euler characteristic is additive with respect to decompositions $X = Z\cup U$ where $Z$ is closed and $U$ is its open complement. Other invariants with this property are called \emph{cut-paste} invariants, or \emph{generalized Euler characteristics}. The Grothendieck ring of varieties $K_0$ provides the universal example of such an invariant: it is  generated by isomorphism classes of varieties, subject to the relation
\[
[X] = [Z]+[X\setminus Z]
\]
whenever $Z$ is closed in $X$, and $[X]\cdot[Y]=[X\times Y]$. Every generalized Euler characteristic with values in some ring $R$ can be identified with a homomorphism $K_0\to R$. The class $[X]$ associated to a variety $X$ is called \emph{the motive of $X$}.  The following equality generalizes a proof by the second author which appeared in Proposition 4.2 of \cite{farb-wolfson}.
Again, we assume that $S$ is nonempty for simplicity; an analogous statement can be made when $S$ is empty.

\begin{theorem}[\textbf{Motive of orbit configuration space}]\label{thm:motive}
Assume that $S\neq\emptyset$.
The motive of the complement $\M(G,X)$ factors as
\begin{equation}
[\M(G,X)] = \prod_{i=0}^{n-1} ([X]-|S|-|G|i).
\end{equation}
\end{theorem}
The universality of the motive now implies a long list of numerical identities, for example:
\begin{itemize}
\item Let $\chi_c$ be the compactly supported Euler characteristic. Then
\begin{equation}\label{eq:eulerchar}
\chi_c(\M(G,X)) =
\prod_{i=0}^{n-1} (\chi_c(X) -|S| - |G|i).
\end{equation}
In particular, when $X$ is smooth, this computes the classical Euler characteristic via Poincar\'{e} duality.
\item The Hodge-Deligne polynomial $\operatorname{HD}(Y)$ is a generalized Euler characteristic, which on closed complex manifolds records the Hodge numbers
\[
\operatorname{HD}(Y) = \sum_{p,q} \dim \Ho^{p,q}(Y)t^p u^q.
\]
Theorem \ref{thm:motive} then gives a formula for $\operatorname{HD}(\M(G,X))$ in terms of $\operatorname{HD}(X)$. 
\item Over finite fields, the number of $\mathbb{F}_q$-points on a variety is a generalized Euler characteristic. Therefore, when the set $S$ is fixed by the Frobenius action, we get the following point-count
\[
\#\M(G,X)(\mathbb{F}_q) = \prod_{i=0}^{n-1} (\# X(\mathbb{F}_q) - |S|-|G|i)
\]
\end{itemize}

\begin{proof}[Proof of Theorem \ref{thm:motive}]
Our goal is to exhibit the following connection between the motive of $\M(G,X)$ and the characteristic polynomial of the poset $\D(G,S)$ (see \S\ref{sec:charpoly}):
\begin{equation}\label{eq:motive-char}
[\M(G,X)] = \sum_{b\in \D(G,S)}\mu(\zero,b)[X]^{n-\rk(b)} = \chi(\D(G,S),[X])
\end{equation}
Recall that the the M\"{o}bius function $\mu(x):=\mu(\zero,x)$ on a poset $P$ is the unique function for which
\begin{equation}\label{eq:mobius}
\mu(\zero) = 1 \; \text{ and } \, \sum_{y\leq x}\mu(y) = 0
\end{equation}
for all $x>\zero$. We apply this property to the formal weighted sum of layers
\[ [M] := \sum_{b\in \D(G,S)}\mu(b) [X^{b}]. \]
In this sum, every point $p\in X^n$ is counted precisely 
\[
\sum_{b \leq b_p}\mu(b)
\]
times, where $b_p$ is the maximal layer that contains $p$ (see \S \ref{sec:localarr}). But then, by Equation \eqref{eq:mobius}, it follows that the only points that contribute to $[M]$ are the ones in $X^{\zero} \setminus \cup_{b>\zero} X^b = \M(G,X)$, and those are counted precisely once. We therefore get
\[
[\M(G,X)] = [M] = \sum_{b\in \D(G,S)}\mu(b) [X^{b}].
\]
Lastly, in \S\ref{sec:layers} we produced isomorphisms $X^{\parts} \cong X^{n-\rk\parts}$ for every layer. Applying the multiplicative relation in $K_0$, we arrive at Equation \eqref{eq:motive-char}. The factorization into linear factors now follows from that of the characteristic polynomial, see Theorem \ref{thm:charpoly}.
\end{proof}

\begin{remark}[\textbf{Fadell-Neuwirth fibration}]
When $X$ is a smooth manifold, the projection that forgets the last point $\M(G,X) \to \M[n-1](G,X)$ realizes the spaces as iterated fiber bundles, with fiber $X$ minus $S$ and $n-1$ many $G$-orbits. Accordingly, our formula for the motive can be read as the recursive expression
\[
[\M(G,X)] = [\M[n-1](G,x)]\cdot \left([X] - |S| - (n-1)|G|) \right)
\]
which behaves as though the motive is multiplicative in bundles. However, this point of view can not provide an alternative proof of Theorem \ref{thm:motive} even for smooth varieties, since the bundles are not locally trivial in the Zariski topology. In fact, they are typically not even locally trivial in the analytic category.
\end{remark}

\begin{example}
\begin{enumerate}
\item The Euler characteristic for the complement of the reflection arrangement
$\A(\Z/d,\C)$ is $(-d)^{n-1}(n-1)!$.
\item The Euler characteristic for the complement of type C toric and elliptic arrangements,
$\A(\Z/2,\C^\times)$ and $\A(\Z/2,E)$, from Example \ref{ex:typeCP}, are 
$(-2)^nn!$ and $(-2)^n(n+1)!$, respectively.
\item The Euler characteristic for the complement of the elliptic arrangement with hexagonal
Dowling poset from Example \ref{ex:Z62} is $(-6)^nn!$.
\item The Euler characteristic for the complement of the elliptic arrangement with square
Dowling poset from Example \ref{ex:Z42} is $(-4)^nn!$.
\end{enumerate}
\end{example}

\begin{remark}
Observe that in each of the examples above, the Euler characteristic is divisible by the order of the corresponding wreath product group $\G$, as one expects from a free action. While this may not be apparent in Equation \eqref{eq:eulerchar}, an equivalent formulation that makes this fact less surprising is the following:
\[ \chi_c(\M(G,X)) = n!|G|^n\binom{\chi_c(X\setminus S)/|G|}{n}\]
noting that the $G$-action on $X\setminus S$ is free.

\end{remark}

%\bibliographystyle{alpha}
%\bibliography{combin_orbit_config}

\begin{thebibliography}{DGP19}

\bibitem[ACH15]{ACH2015}
F.~Ardila, F.~Castillo, and M.~Henley.
\newblock The arithmetic {T}utte polynomials of the classical root systems.
\newblock {\em Int. Math. Res. Not. IMRN}, 2015(12):3830--3877, 2015.

\bibitem[Arn69]{Arnold1969}
V.~I. Arnol'd.
\newblock The cohomology ring of the colored braid group.
\newblock {\em Mat. Zametki}, 5:227--231, 1969.

\bibitem[BB90]{BB1990}
H.~Barcelo and N.~Bergeron.
\newblock The {Orlik}-{Solomon} algebra on the partition lattice and the free
  {Lie} algebra.
\newblock {\em J. Combin. Theory Ser. A}, 55(1):80--92, 1990.

\bibitem[BG18]{FPSAC}
C.~Bibby and N.~Gadish.
\newblock Combinatorics of orbit configuration spaces.
\newblock {\em S\'{e}m. Lothar. Combin.}, 80B:Art. 72, 11, 2018.

\bibitem[BG19]{BG2}
C.~Bibby and N.~Gadish.
\newblock A generating function approach to new representation stability
  phenomena in orbit configuration spaces.
\newblock {\em arXiv preprint}, 1911.02125, 2019.

\bibitem[Bib16]{bibby1}
C.~Bibby.
\newblock Cohomology of abelian arrangements.
\newblock {\em Proc. Amer. Math. Soc.}, 144(7):3093--3104, 2016.

\bibitem[Bib18]{bibby2}
C.~Bibby.
\newblock Representation stability for the cohomology of arrangements
  associated to root systems.
\newblock {\em J. Algebraic Combin.}, 48(1):51--75, 2018.

\bibitem[Bri73]{Brieskorn1973}
E.~Brieskorn.
\newblock Sur les groupes de tresses [d'apr\'es {V}. {I}. {A}rnol'd].
\newblock In {\em S\'eminaire {B}ourbaki, 24\`eme ann\'ee (1971/1972), {E}xp.
  {N}o. 401}, pages 21--44. Lecture Notes in Math., Vol. 317. Springer, Berlin,
  1973.

\bibitem[Cas16]{Casto2016}
K.~Casto.
\newblock {FI}$_{G}$-modules, orbit configuration spaces, and complex
  reflection groups.
\newblock {\em arXiv preprint}, 1608.06317, 2016.

\bibitem[DGP19]{DGP2017}
E.~Delucchi, N.~Girard, and G.~Paolini.
\newblock Shellability of posets of labeled partitions and arrangements defined
  by root systems.
\newblock {\em Electron. J. Combin.}, 26(4):Paper No. 4.14, 22, 2019.

\bibitem[Dow73]{Dowling1973}
T.~A. Dowling.
\newblock A class of geometric lattices based on finite groups.
\newblock {\em J. Combin. Theory Ser. B}, 14:61--86, 1973.

\bibitem[DS18]{DS}
G.~Denham and A.~I. Suciu.
\newblock Local systems on complements of arrangements of smooth, complex
  algebraic hypersurfaces.
\newblock {\em Forum Math. Sigma}, 6:Paper No. e6, 20, 2018.

\bibitem[Dup15]{dupont}
C.~Dupont.
\newblock The {Orlik}-{Solomon} model for hypersurface arrangements.
\newblock {\em Ann. Inst. Fourier (Grenoble)}, 65(6):2507--2545, 2015.

\bibitem[Fol66]{Folkman1966}
J.~Folkman.
\newblock The homology groups of a lattice.
\newblock {\em J. Math. Mech.}, 15:631--636, 1966.

\bibitem[FW16]{farb-wolfson}
B.~Farb and J.~Wolfson.
\newblock Topology and arithmetic of resultants, {I}.
\newblock {\em New York J. Math.}, 22:801--821, 2016.

\bibitem[FZ02]{FZ}
E.~M. Feichtner and G.~M. Ziegler.
\newblock On orbit configuration spaces of spheres.
\newblock {\em Topology Appl.}, 118:85--102, 2002.

\bibitem[GM88]{GM}
M.~Goresky and R.~MacPherson.
\newblock {\em Stratified {M}orse theory}, volume~14 of {\em Ergebnisse der
  Mathematik und ihrer Grenzgebiete (3) [Results in Mathematics and Related
  Areas (3)]}.
\newblock Springer-Verlag, Berlin, 1988.

\bibitem[GW00]{GW2000}
E.~Gottlieb and M.~L. Wachs.
\newblock Cohomology of {Dowling} {Lattices} and {Lie} ({Super}){Algebras}.
\newblock {\em Adv. in Appl. Math.}, 24(4):301--336, 2000.

\bibitem[Hal17]{Hallam2014}
J.~Hallam.
\newblock Applications of {Quotient} {Posets}.
\newblock {\em Discrete Math.}, 340(4):800--810, 2017.

\bibitem[Han84]{Hanlon1984}
P.~Hanlon.
\newblock The characters of the wreath product group acting on the homology
  groups of the {D}owling lattices.
\newblock {\em J. Algebra}, 91(2):430--463, 1984.

\bibitem[Loo76]{looijenga}
E.~Looijenga.
\newblock Root systems and elliptic curves.
\newblock {\em Invent. Math.}, 38(1):17--32, 1976.

\bibitem[LS86]{LS1986}
G.~I. Lehrer and L.~Solomon.
\newblock On the action of the symmetric group on the cohomology of the
  complement of its reflecting hyperplanes.
\newblock {\em J. Algebra}, 104(2):410--424, 1986.

\bibitem[Moc08]{Moci2008}
L.~Moci.
\newblock Combinatorics and topology of toric arrangements defined by root
  systems.
\newblock {\em Atti Accad. Naz. Lincei Rend. Lincei Mat. Appl.},
  19(4):293--308, 2008.

\bibitem[Moc12]{Moci2012}
L.~Moci.
\newblock A {T}utte polynomial for toric arrangements.
\newblock {\em Trans. Amer. Math. Soc.}, 364(2):1067--1088, 2012.

\bibitem[Nak83]{Nakamura1983}
T.~Nakamura.
\newblock A note of the $k(\pi,1)$-property of the orbit space of the unitary
  reflection group $g(m,l,n)$.
\newblock {\em Sci. Papers College of Arts and Sciences, Univ. Tokyo}, 33:1--6,
  1983.

\bibitem[OS80]{OS1980}
P.~Orlik and L.~Solomon.
\newblock Combinatorics and topology of complements of hyperplanes.
\newblock {\em Invent. Math.}, 56(2):167--189, 1980.

\bibitem[Pao20]{paolini}
G.~Paolini.
\newblock Shellability of generalized {D}owling posets.
\newblock {\em J. Combin. Theory Ser. A}, 171:105159, 19, 2020.

\bibitem[Pet17]{Petersen2017}
D.~Petersen.
\newblock A spectral sequence for stratified spaces and configuration spaces of
  points.
\newblock {\em Geom. Topol.}, 21(4):2527--2555, 2017.

\bibitem[Sta72]{Stanley1972}
R.~P. Stanley.
\newblock Supersolvable lattices.
\newblock {\em Algebra Universalis}, 2(1):197, 1972.

\bibitem[Sta82]{Stanley1982}
R.~P Stanley.
\newblock Some aspects of groups acting on finite posets.
\newblock {\em J. Combin. Theory Ser. A}, 32(2):132--161, 1982.

\bibitem[Tot96]{Totaro1996}
B.~Totaro.
\newblock Configuration spaces of algebraic varieties.
\newblock {\em Topology}, 35(4):1057--1067, 1996.

\bibitem[Wac98]{Wachs1998}
M.~L. Wachs.
\newblock On the (co)homology of the partition lattice and the free {Lie}
  algebra.
\newblock {\em Discrete Math.}, 193(1):287--319, 1998.

\bibitem[Xic97]{Xthesis}
M.~A. Xicot{\'{e}}ncatl.
\newblock {\em Orbit configuration spaces, infinitesimal braid relations in
  homology and equivariant loop spaces}.
\newblock ProQuest LLC, Ann Arbor, MI, 1997.
\newblock Thesis (Ph.D.)--University of Rochester.

\bibitem[Xic02]{X}
M.~A. Xicot{\'{e}}ncatl.
\newblock Product decomposition of loop spaces of configuration spaces.
\newblock {\em Topology Appl.}, 121(1):33--38, 2002.

\end{thebibliography}

\end{document}